\documentclass[10pt]{amsart}

\usepackage{amssymb, amsmath, color}
\usepackage{enumitem}
\usepackage{graphicx}
\usepackage{tikz, caption, subcaption}
\usepackage[pdftex]{hyperref}

\numberwithin{equation}{section}

\newtheorem{thm}{Theorem}[section]
\newtheorem{cor}[thm]{Corollary}
\newtheorem{lem}[thm]{Lemma}
\newtheorem{defn}[thm]{Definition}
\newtheorem{prop}[thm]{Proposition}

\theoremstyle{remark}
\newtheorem{rem}{Remark}

\newcommand{\SO}{\mathrm{SO}}

\newcommand{\dist}{\mathrm{dist}}
\newcommand{\supp}{\mathrm{supp\,}}
\newcommand{\real}{\mathrm{Re\,}}
\newcommand{\imag}{\mathrm{Im\,}}
\newcommand{\sgn}{\mathrm{sgn}}
\newcommand{\p}{\partial}

\renewcommand{\div}{\mathrm{div\,}}

\begin{document}

\title[Resolvent estimates]{Resolvent estimates for the Lam\'e operator \\
and failure of Carleman estimates 
}

\author[Y. Kwon]{Yehyun Kwon}
\author[S. Lee]{Sanghyuk Lee}
\author[I. Seo]{Ihyeok Seo}

\address{(Yehyun Kwon) School of Mathematics, Korea Institute for Advanced Study,  Seoul 02455, Republic of Korea}
\email{yhkwon@kias.re.kr}
\address{(Sanghyuk Lee) Department of Mathematical Sciences and RIM, Seoul National University, Seoul 08826, Republic of Korea}
\email{shklee@snu.ac.kr}
\address{(Ihyeok Seo) Department of Mathematics, Sungkyunkwan University, Suwon 16419, Republic of Korea
}
\email{ihseo@skku.edu}


\begin{abstract}
In this paper, we consider  the Lam\'e operator $-\Delta^\ast$ and study  resolvent estimate, uniform Sobolev estimate, and Carleman  estimate for $-\Delta^\ast$. First, we obtain sharp $L^p$--$L^q$ resolvent estimates for $-\Delta^\ast$ for admissible $p,q$.   This extends the particular case $q=\frac p{p-1}$ due to Barcel\'o et al. \cite{BFPRV} and Cossetti \cite{Co19}.  Secondly, we show failure of uniform Sobolev estimate and Carleman estimate for $-\Delta^\ast$. For the purpose we  directly analyze the Fourier multiplier of the resolvent. This allows us to prove not only the upper bound but also  the lower bound on the resolvent, so we get  the sharp $L^p$--$L^q$ bounds for the resolvent of $-\Delta^\ast$. Strikingly, the relevant  uniform Sobolev and Carleman estimates turn out to be false for the Lam\'e operator $-\Delta^\ast$ even though  the uniform resolvent estimates for $-\Delta^\ast$ are valid for certain range of $p, q$. This contrasts with the classical result regarding the Laplacian  $\Delta$ due to  Kenig, Ruiz, and Sogge \cite{KRS87} in which  the uniform resolvent estimate plays crucial role in proving the uniform Sobolev and Carleman  estimates for $\Delta$. We also describe locations of the $L^q$-eigenvalues of $-\Delta^\ast+V$ with complex potential $V$ by making use of the sharp $L^p$--$L^q$ resolvent estimates for $-\Delta^\ast$. 
\end{abstract}
\maketitle

\section{Introduction}
Let $-\Delta^\ast$ be the Lam\'e operator acting on $\mathcal S(\mathbb R^d)^d$ which is given  by
\[	-\Delta^\ast u  :=-\mu \Delta u -(\lambda+\mu) \nabla \div u, \quad u \in \mathcal S(\mathbb R^d)^d. \]
Here, the Lam\'e coefficients $\lambda$ and $\mu$ are real numbers satisfying
\[	\mu>0, \quad \lambda+2\mu>0, \]
and $\mathcal S(\mathbb R^d)^d$ denotes the space of all $d$-tuples of complex-valued Schwartz functions  on $\mathbb R^d$. When $d=3$, the Lam\'e operator has significant role in describing a linear homogeneous and isotropic elastic medium, and  in such case $u$ denotes the displacement field of the medium. For more about  physical and mathematical backgrounds of the operator, see, for example, \cite[pp. 1023--1033]{Gr94}, \cite{LL}, \cite{MH83}, and \cite{NH}.

In this paper, we are concerned with the following  $L^p$--$L^q$ resolvent estimates for $-\Delta^\ast$;
\begin{equation}\label{e:resol2}
	\|(-\Delta^\ast-z)^{-1}f \|_{L^q(\mathbb R^d)^d}\le C\kappa_{p,q}(z)\|f\|_{L^p(\mathbb R^d)^d}, \quad \forall f\in \mathcal S(\mathbb R^d)^d,
\end{equation}
for admissible pairs of $p,q$ and spectral parameters $z$ in the resolvent set $\rho(-\Delta^\ast):=\mathbb C\setminus \sigma(-\Delta^\ast)=\mathbb C\setminus [0,\infty)$. Here, $\kappa_{p,q}\colon \rho(-\Delta^\ast)\to \mathbb R$ is a positive function and $C$ is a constant independent of $z\in \rho(-\Delta^\ast)$.  Moreover, we also show the sharpness of the bound $\kappa_{p,q}(z)$ in  \eqref{e:resol2} up to a multiplicative constant.

As is to be seen below, the bound in \eqref{e:resol2} not only  superficially  resembles  the $L^p$--$L^q$ resolvent estimate for the Laplacian $-\Delta$ but also  share  similar characteristics with that of the Laplacian. So, we begin with a brief discussion of the resolvent estimates for $-\Delta$ and relevant previous results including  the uniform Sobolev and Carleman estimates.

\subsubsection*{Resolvent estimate for the Laplacian} 
The $L^p$--$L^q$ resolvent estimate for the Laplacian $-\Delta$ is referred to the following form of  \emph{a priori} inequality
\begin{equation}\label{e:laplace}
	\|(-\Delta-z)^{-1}\|_{p\to q}\le C\kappa_{p,q}(z), 
\end{equation}
where $\|T\|_{p\to q}$ denotes
\begin{equation}\label{o-norm}
\|T\|_{p\to q}:=\inf \big\{K\colon \|Tf\|_{L^q(\mathbb R^d)}\le K\|f\|_{L^p(\mathbb R^d)}, \, \forall f\in \mathcal S(\mathbb R^d) \big\},
\end{equation}
$\kappa_{p,q} $ is a positive function defined on the resolvent set  $\rho(-\Delta):=\mathbb C\setminus \sigma(-\Delta)=\mathbb C\setminus [0,\infty)$, and $C=C_{p,q,d}$ is a constant independent of $z\in\rho(-\Delta)$.  The first result on \eqref{e:laplace} goes back to the seminal work of Kenig, Ruiz, and Sogge \cite{KRS87} in which they used the estimate for study of unique continuation. They proved that the estimate \eqref{e:laplace} holds with $\kappa_{p,q}(z)\equiv 1$  if and only if  $\frac1p-\frac1q=\frac2d$ and $\frac{2d}{d+3}<p<\frac{2d}{d+1}$, which is equivalent to the condition that $(\frac1p,\frac1q)\in(A,A')$. See Figures \ref{figthm2} and \ref{figthm}, and Definition \ref{defi} below. Later, the range of $p,q$ was extended by Guti\'{e}rrez \cite{Gu04} to $\frac{2}{d+1}\le\frac1p-\frac1q\le \frac2d$,  $\frac1p>\frac {d+1}{2d}$ and  $\frac1q<\frac{d-1}{2d}$ (the range $\mathcal R_1$). In this range, the bound takes the form $\kappa_{p,q}(z)=|z|^{-1+\frac d2(\frac 1p-\frac 1q)}$, and the bound is independent of the distance between $z$ and the spectrum $\sigma(-\Delta)$.

Recently, two of the authors  \cite{KL19} extended the resolvent estimate outside the uniform boundedness range. More precisely, they proved \eqref{e:laplace} for general pairs of $(\frac1p,\frac1q)\in\mathcal R_1\cup \widetilde{\mathcal R}_2\cup\widetilde{\mathcal R}_3\cup\widetilde{\mathcal R}_3'$ (see Definition \ref{defi} and Figures \ref{figthm2} and \ref{figthm} below for precise description) with
\begin{equation}\label{e:resol-bound}
	\kappa_{p,q}(z) := |z|^{-1+\frac d2(\frac 1p-\frac 1q)}  \dist(z/|z|,[0,\infty))^{-\gamma_{p,q}}
\end{equation}
and 
\begin{equation}\label{e:bounds}
	\gamma_{p,q} : =\max \left\{ 0, ~ 1-\frac{d+1}2\Big(\frac 1p-\frac 1q \Big),~ \frac{d+1}2-\frac dp, ~ \frac dq-\frac{d-1}{2} \right\}.
\end{equation}
We also refer to \cite[Proposition 2.1]{Fr18} for $L^p$--$L^{p'}$ estimates (in a dual form) which can be deduced by interpolation between the trivial $L^2$--$L^2$ estimate and the $L^\frac{2(d+1)}{d+3}$--$L^\frac{2(d+1)}{d-1}$ estimate due to Kenig, Ruiz, and Sogge \cite{KRS87}.  Outside the  range $\mathcal R_1$, the bound $\kappa_{p,q}(z)$ depends  not only on $|z|$ but also on $\dist (z,[0,\infty))$,  and it exhibits singular behavior as the spectral parameter $z$ approaches to the spectrum $\sigma(-\Delta)=[0,\infty)$.  Moreover,  it is also proven in \cite{KL19} that the estimate \eqref{e:laplace} is sharp in the sense that the inequality is reversed when $C$ is replaced with some smaller constant. The sharp resolvent estimate \eqref{e:laplace} was used to  characterize various profiles of the spectral region which contains $L^q$-eigenvalues of the non-self-adjoint operators $-\Delta+V$ with complex-valued potential $V$.

\subsubsection*{Uniform Sobolev inequality and Carleman estimate for $-\Delta$}
In \cite{KRS87}, making use of the uniform resolvent estimate, the authors established the uniform Sobolev inequality 
\begin{equation}\label{e:unif-lapl}
	\|u\|_{L^q(\mathbb R^d)}\le C\|(-\Delta +a\cdot\nabla +b)u\|_{L^p(\mathbb R^d)}
\end{equation}
with $C$ independent of $(a,b)\in\mathbb C^d\times \mathbb C$ whenever $(\frac1p,\frac1q)\in(A,A')$. This immediately gives the following type of Carleman estimate
\begin{equation}\label{e:car-lapl}
	\|u\|_{L^q(\mathbb R^d)}\le C\|e^{v\cdot x}(-\Delta)e^{-v\cdot x}u\|_{L^p(\mathbb R^d)}
\end{equation}
with $C$ independent of $v\in\mathbb R^d$ for the same range of $p,q$.  This type of Carleman estimate \eqref{e:car-lapl} was  used to obtain unique continuation property of the differential inequality $|\Delta u|\le |Vu|$ for $V\in L^\frac d2(\mathbb R^d)$ (see \cite[Theorem 3.1 and Corollaries 3.1 and 3.2]{KRS87}). Also, see \cite{JKL16, JKL18, CS90, S14, S16} for related results.

In \cite{JKL18} it was shown that, for $d\ge 3$ and $1<p,q<\infty$, the Carleman estimate \eqref{e:car-lapl} holds if and only if 
\begin{equation}\label{e:car-iff}
	\frac 1p-\frac 1q=\frac 2d,   \quad  \frac{d^2-4}{2d(d-1)}  \le \frac1p \le \frac{d+2}{2(d-1)}.	
\end{equation}
The range of $p,q$ in \eqref{e:car-iff}, which properly contains that of $(1/p,1/q)\in (A,A')$, is optimal for the Carleman estimate \eqref{e:car-lapl} when $d\ge 4$ (\cite[Theorem 1.1]{JKL18}). This exhibits different natures of the uniform resolvent estimates and the Carleman estimates.  Such difference in the boundedness is attributed to different size of sets which  carry singularities of the relevant Fourier multipliers. See  \cite{JKL18} for more details.

\subsubsection*{Notations} 
In order to facilitate the statement of our results, we introduce some notations. We use the following norms in the vector-valued setting: For a vector-valued function $u = (u_1, \ldots ,u_d)$ let us set
\[	\|u\|_{L^q(\mathbb R^d)^d} 
	= 
	\begin{cases}
		\big(\sum_{j=1}^d\|u_j\|_{L^q(\mathbb R^d)}^q\big)^{\frac1q}, & 1\le q<\infty, \\
		\max_{1\le j\le d}\|u_j\|_{L^\infty(\mathbb R^d)}, & q=\infty,
	\end{cases}
\]
and define the Lorentz norm $\|u\|_{L^{q,r}(\mathbb R^d)^d}$ similarly. For $T=(-\Delta^\ast-z)^{-1}$ we define $\|T\|_{p\to q}$ in the same way as in \eqref{o-norm} replacing $L^p(\mathbb R^d)$ with $L^p(\mathbb R^d)^d$. Next, we recall the notations from \cite{KL19}. We record them below for the reader's convenience and to make this article self-contained.  For the cases $d=3,4$, referring to Figures \ref{figthm2} and \ref{figthm}  can be helpful for the reader to follow the definitions and notations below (see \cite[Figs. 3 and 4]{KL19} for $d=2$ and $d\ge 5$). Also, the interested readers are encouraged to refer to \cite{KL19} for details regarding  $P_\circ$, $P_\ast$ and the regions $\widetilde{\mathcal R}_2$, $\widetilde{\mathcal R}_3$ which are defined below.

\begin{figure}
\captionsetup{type=figure,font=footnotesize}
\begin{minipage}[b]{0.5\textwidth}
\centering
\begin{tikzpicture} [scale=0.6]\scriptsize
	\path [fill=lightgray] (0,0)--(3,3)--(4, 3)--(5,5)--(7,6)--(7,7)--(10,10)--(10,10-20/3)--(20/3,0)--(0,0);
	\draw [<->] (0,10.7)node[left]{$y$}--(0,0) node[below]{$(0,0)$}--(10.7,0) node[right]{$x$};
	\draw (0,10) --(10,10)--(10,0) node[below]{$(1,0)$};
	\draw (10/3,10/3)node[above]{$D$}--(20/3,5/3)node[above]{$B$}--(10-5/3,10-20/3)node[left]{$B'$}--(10-10/3,10-10/3)node[left]{$D'$};
	\draw (0,0)--(10/3,10/3);
	\draw [dash pattern={on 2pt off 1pt}] (10/3,10/3)--(10-10/3,10-10/3);
	\draw (10-10/3,10-10/3)--(10,10);
	\draw (20/3,0)--(10,10-20/3);
	\draw [dash pattern={on 2pt off 1pt}]  (0,5)node[left]{$\frac12$}--(5,5)node[above]{$H$}--(5,0)node[below]{$\frac12$}; 
	\draw [dash pattern={on 2pt off 1pt}] (20/3, 5/3)--(20/3, 0)node[below]{$A$};
	\draw [dash pattern={on 2pt off 1pt}] (10-5/3, 10-20/3)--(10, 10-20/3)node[right]{$A'$};
	\draw [dash pattern={on 2pt off 1pt}] (4,3)node[right] {$P_\circ$}--(5,5)node[above]{$H$}--(7,6)node[below]{$P_\circ'$}--(7,7)node[above]{$P_\ast'$};
	\draw (6.2, 3.8) node{$\widetilde{\mathcal R}_2$};
	\draw (30/8, 3/2) node{$\widetilde{\mathcal R}_3$};
	\draw (10-3/2, 50/8) node{$\widetilde{\mathcal R}_3'$};
	\draw (7.9, 2.1) node{$\mathcal R_1$};
	\draw [dash pattern={on 2pt off 1pt}] (3,3)node[left]{$P_*$}--(4,3);
\end{tikzpicture}\caption{The case $d=3$}\label{figthm2}
\end{minipage}\hfill
\begin{minipage}[b]{0.5\textwidth}
\centering
\begin{tikzpicture} [scale=0.6]\scriptsize
	\path [fill=lightgray] (0,0)--(25/7, 25/7)--(110/26,90/26)--(5,5)--(10-90/26,10-110/26)--(10-25/7,10-25/7)--(10,10)--(10,5)--(5,0)--(0,0);
	\draw [<->] (0,10.7)node[left]{$y$}--(0,0) node[below]{$(0,0)$}--(10.7,0) node[right]{$x$};
	\draw (0,10) --(10,10)--(10,0) node[below]{$(1,0)$};
	\draw (30/8,30/8)node[above]{$D$}--(25/4,9/4)node[above]{$B$}--(10-9/4,10-25/4)node[left]{$B'$}--(10-30/8,10-30/8)node[left]{$D'$};
	\draw (0,0)--(30/8,30/8);
	\draw [dash pattern={on 2pt off 1pt}] (30/8,30/8)--(10-30/8,10-30/8);
	\draw (10-30/8,10-30/8)--(10,10);
	\draw (5,0)--(10,5);
	\draw [dash pattern={on 2pt off 1pt}]  (0,5)node[left]{$\frac12$}--(5,5)node[above]{$H$}--(5,0)node[below]{$\frac12$}; 
	\draw [dash pattern={on 2pt off 1pt}] (50/8, 9/4)--(50/8, 5/4)node[below]{$A$}--(50/8,0)node[below]{$E$};
	\draw [dash pattern={on 2pt off 1pt}] (10-9/4,10-50/8)--(10-5/4, 10-50/8)node[right]{$A'$}--(10,10-50/8)node[right]{$E'$};
	\draw [dash pattern={on 2pt off 1pt}] (110/26, 90/26)node[below] {$P_\circ$}--(5,5)node[above]{$H$}--(10-90/26,10-110/26)node[right]{$P_\circ'$}--(10-50/14,10-50/14)node[above]{$P_\ast'$};
	\draw (6, 4) node{$\widetilde{\mathcal R}_2$};
	\draw (30/8, 1.8) node{$\widetilde{\mathcal R}_3$};
	\draw (8.2, 50/8) node{$\widetilde{\mathcal R}_3'$};
	\draw (7.25, 2.75) node{$\mathcal R_1$};
	\draw [dash pattern={on 2pt off 1pt}] (50/14,50/14)node[left]{$P_*$}--(110/26, 90/26);
\end{tikzpicture}\caption{The case $d=4$}\label{figthm}
\end{minipage}
\end{figure}

\begin{defn}\label{defi}
Let $I^2 =\{(x,y)\in\mathbb R^2\colon 0\le x,y\le 1\}$. For each $(x,y) \in I^2$ we set $(x,y)' =(1-y, 1-x)$. Similarly, for $\mathcal R\subset I^2$ we define $\mathcal R'\subset I^2$ by $\mathcal R' =\{ (x,y)\in I^2\colon (x,y)' \in \mathcal R \}$. For $X_1, \ldots, X_m\in I^2$, we denote by $[X_1, \cdots, X_m]$ the convex hull of  the points $X_1, \ldots, X_m$. In particular, if  $X, Y\in I^2$,  $[X,Y]$  denotes  the closed line segment  connecting $X$ and $Y$ in $I^2$. We also denote by $(X,Y)$ and $[X,Y)$  the open interval $[X,Y]\setminus\{X,Y\}$ and the half-open interval $[X,Y]\setminus\{Y\}$, respectively. 
\begin{itemize}[leftmargin=*]
\item \emph{\emph{\bf Points and lines in $I^2$ :}} 
Let $\mathcal L=\{(x,y)\in I^2\colon y=\frac{d-1}{d+1}(1-x)\}$.  For $d\ge3$ let us denote by $A\in I^2$ the intersection of the lines $\mathcal L$ and $x-y=\frac2d$, that is, $A=(\frac{d+1}{2d}, \frac{d-3}{2d})$.  For $d\ge 2$ let $B\in I^2$ be the intersection of the lines $\mathcal L$ and $x-y=\frac2{d+1}$, that is, $B=(\frac{d+1}{2d}, \frac{(d-1)^2}{2d(d+1)})$.\footnote{Consequently, the line segments $[A,A']$ and $[B,B']$ are contained in the lines $x-y=\frac{2}{d}$ and $x-y=\frac 2{d+1}$, respectively.} The point $D=(\frac{d-1}{2d}, \frac{d-1}{2d})$ is the intersection of the diagonal $x=y$ and $\mathcal L$, and the point $E=(\frac{d+1}{2d},0)$ is the projection of $A$ (or $B$) onto the $x$-axis. Also, we set $H=(\frac12, \frac12)$, and define the points $P_\ast=(\frac1{p_\ast},\frac1{p_\ast})$ and $P_\circ=(\frac1{p_\circ}, \frac1{q_\circ})$ by\footnote{The number $p_\ast$ relates to the range of  the oscillatory integral of Carleson--Sj\"olin type with elliptic phase (\cite[Theorem 1.2]{GHI}). Being combined with Tao's bilinear restriction theorem (\cite{Ta03}) and the bilinear  argument in \cite{CKLS, Lee04}, this is one of main ingredients for the results in \cite[Theorem 1.4]{KL19}. The point $P_\circ$ is the intersection of $\mathcal L$ and the line connecting $P_\ast$ and $(\frac12, \frac{d}{2(d+2)})$; see \cite[Sections 2 and 3]{KL19} for details. If $d=2$ then $P_\ast=P_\circ=D=(1/4,1/4)$.} 
	\begin{align*}
	\frac1{p_*} := 
		\begin{cases}
			\frac{3(d-1)}{2(3d+1)}
				&\!\!\!\text{ if } d \text{ is odd,} \\
		 	\frac{3d-2}{2(3d+2)}
				&\!\!\!\text{ if } d \text{ is even,}
		\end{cases}    
	\quad  
	\Big(\frac1{p_\circ}, \frac1{q_\circ}\Big) :=
		\begin{cases}
			\big( \frac{(d+5)(d-1)}{2(d^2+4d-1)},\, \frac{(d-1)(d+3)}{2(d^2+4d-1)} \big) 
				&\!\!\!\text{ if } d \text{ is odd,} \\
			\big( \frac{d^2+3d-6}{2(d^2+3d-2)}, \, \frac{(d-1)(d+2)}{2(d^2+3d-2)} \big)    
				&\!\!\!\text{ if } d \text{ is even.}
		\end{cases}
	\end{align*} 
\item \emph{\bf{Regions in $I^2$} :} $\widetilde{\mathcal R}_2=[B,B',P_\circ ', H, P_\circ]\setminus\big( [P_\circ, H) \cup [P_\circ ', H)\cup [B, B'] \big)$, 
	\begin{align*}
	\mathcal R_1
	&=	\begin{cases}
			[B,E,B',E',(1,0)]\setminus ([B,E]\cup [B',E']\cup\{(1,0)\}) &\text{ if } d=2,\\
			[A,B,A',B']\setminus ([A,B]\cup[A',B']) &\text{ if } d\ge3,
		\end{cases} \\
	\widetilde{\mathcal R}_3 
	&=	\begin{cases}
			[(0,0),E,B,D]\setminus ([B,D]\cup[B,E]) &\text{ if } d=2, \\
			[(0,0), (\frac2d,0), A, B, P_\circ, P_\ast]\setminus ([A,B]\cup [B,P_\circ]\cup[P_\circ,P_\ast]\cup\{(\frac2d,0)\}) & \text{ if } d\ge3.
		\end{cases}
	\end{align*}
\item \emph{\bf{Spectral regions} : }For $p,q$ satisfying  $(\frac1p,\frac1q)\in \mathcal R_1 \cup \big( \bigcup_{i=2}^3\widetilde{\mathcal R}_i\big) \cup\widetilde{\mathcal R}_3'$, and $\ell>0$ we define the region $\mathcal Z_{p,q}(\ell)$ in the complex plane by
\[	\mathcal Z_{p,q}(\ell)= \{z\in\mathbb C\setminus [0,\infty)\colon \kappa_{p,q}(z)\le \ell\}.	\] 
\end{itemize}
\end{defn}
Various profiles of $\mathcal Z_{p,q}(\ell)$ depending on $p,q,d,$ and $\ell$ can be found in \cite{KL19}.

\subsubsection*{Resolvent estimate for the Lam\'e operator} 
If $d=1$ the Lam\'e operator is just a constant times the Laplacian, that is, $\Delta^\ast = (\lambda+2\mu)\frac{d^2}{dx^2}$. Also, in every dimension, if $\mu=\lambda+2\mu$ then $\Delta^\ast=\mu(\Delta,\ldots, \Delta)$. In these cases, the resolvent estimate \eqref{e:resol2}  trivially  follows  from the estimate \eqref{e:laplace} regarding  the Laplacian.  Hence, throughout the paper, we shall assume that $d\ge 2$ and $\mu\neq \lambda+2\mu$.

Barcel\'o, Folch-Gabayet, P\'erez-Esteva, Ruiz, and Vilela  \cite[Theorem 1.1]{BFPRV} showed that $\|(-\Delta^\ast-z)^{-1}\|_{p\to q}\lesssim |z|^{-1+\frac d2(\frac 1p-\frac 1q)}$ for $p,q$ satisfying $\frac1p+\frac1q=1$ and $\frac2{d+1}\le\frac1p-\frac1q\le\frac2d$, that is, when $(\frac1p, \frac1q)$ lies on the closed line segment of which endpoints are the midpoints of $[A,A']$ and $[B,B']$.  Also, see \cite[Theorem 2.3]{Co19}.  They utilized the Leray projection $\Pi$ defined via the Fourier transform by
\[	\widehat{\Pi f}(\xi) = \left(\widehat f(\xi)\cdot \frac{\xi}{|\xi|}\right)\frac{\xi}{|\xi|} 	\] 
to decompose the Lam\'e resolvent as follows:
\begin{equation}\label{e:helmholtz}
	(-\Delta^\ast-z)^{-1} = (-\mu\Delta-z)^{-1}(I_d-\Pi )  +(-(\lambda+2\mu)\Delta-z)^{-1} \Pi ,
\end{equation}
where $I_d$ denotes the $d\times d$ identity matrix. The estimate for $(-\Delta^\ast-z)^{-1}$ is now a simple consequence of the corresponding estimate for $(-\Delta-z)^{-1}$ and boundedness of the Riesz transforms.  Indeed, applying the known estimate \eqref{e:laplace} (\cite{KL19}) one can get the upper bounds on $\|(-\Delta^\ast-z)^{-1}\|_{p\to q}$ for $p,q$ contained in a range which is  wider than that of \cite{BFPRV, Co19}. See Theorem \ref{t:resol}.

This immediately leads us to a couple of related questions which are already known to  be true for the Laplacian. First, one may ask  \emph{whether  these bounds are sharp.}  Secondly,  in point of view of the above \eqref{e:helmholtz} \emph{it seems likely that  the uniform Sobolev  inequality  \eqref{e:usi} and Carleman estimate \eqref{e:car} are also possible} since \eqref{e:helmholtz}  and the known estimates for $(-\Delta-z)^{-1}$  give uniform resolvent estimate \eqref{e:resol2} for $p,q$ satisfying $(\frac1p,\frac1q)\in(A,A')$ (see Theorem \ref{t:resol} below).

Now we begin stating  our results on the sharp resolvent estimates for $-\Delta^\ast$. By the following proposition, we cannot expect the resolvent estimate \eqref{e:resol2} when $(\frac1p,\frac1q)$ lies outside the admissible range:
\[	\mathcal R_0:= 
	\begin{cases}
		\big\{ (x,y)\in I^2\colon 0\le x-y <1 \big\}
		   &\text{ if } ~  d=2,\\
		\big\{(x,y)\in I^2\colon 0\le x-y \le  \frac2d\big\} \setminus \big\{ \big(1, \frac{d-2}d \big) , \big(\frac 2d, 0\big) \big\}	
		   &\text{ if } ~  d\ge 3.
	\end{cases}	\]
\begin{prop}\label{p:admissible}
	Let $d\ge 2$ and let $1\le p, q\le \infty$. If $(\frac1p, \frac1q) \notin \mathcal R_0$, then  for any $z\in\mathbb C\setminus [0,\infty)$ $\|(-\Delta^\ast-z)^{-1}\|_{p\to q}=\infty$.  
\end{prop}

In what follows, we characterize $L^p$--$L^q$ resolvent estimates for the Lam\'e operator for a large set  of the admissible $(\frac1p,\frac1q)\in \mathcal R_0$.
\begin{thm}\label{t:resol}
Let $d\ge 2$, $z\in\mathbb C\setminus [0,\infty)$, $1<p\le q< \infty$. If $(\frac1p, \frac1q) \in \mathcal R_1\cup\widetilde{\mathcal R}_2 \cup \widetilde{\mathcal R}_3 \cup\widetilde{\mathcal R}_3'$, we have
\begin{equation}\label{e:shp-resol}
	C^{-1}\kappa_{p,q}(z)\le \|(-\Delta^\ast-z)^{-1}\|_{p\to q} \le C \kappa_{p,q}(z),
\end{equation}
where the  constant $C$ may depend on $p$, $q$, $d$, $\lambda$, and $\mu$, but is independent of $z\in \mathbb C\setminus [0,\infty)$. Furthermore, we have the following  weak type and restricted weak type estimates in the critical cases:
\begin{gather}\label{e:rest-wk-type}
	\|(-\Delta^\ast-z)^{-1}f\|_{L^{q,\infty}(\mathbb R^d)^d} \lesssim \|f\|_{L^{p,1}(\mathbb R^d)^d} \quad \text{if} \quad (1/p,1/q)=B ~ \text{or} ~ B',
\end{gather}
and
\begin{gather}\label{e:wk-type}
	\|(-\Delta^\ast-z)^{-1}f\|_{L^{q,\infty}(\mathbb R^d)^d} \lesssim \|f\|_{L^p(\mathbb R^d)^d} \quad \text{if} \quad 
	\begin{cases}
		(1/p,1/q)\in (B', E'), \!\!  &\text{when} ~ d=2,\\
		(1/p,1/q)\in (B',A'], \!\! &\text{when} ~ d\ge3.
	\end{cases}
\end{gather}	
\end{thm}

Throughout the paper, $A\lesssim B$ denotes $A\le CB$ for some constant $C>0$, and  $A\approx B$ denotes $A\lesssim B\lesssim A$. The lower bound in \eqref{e:shp-resol} also holds for all $1\le p,\,q\le \infty$. See Section \ref{sec:2} for the details. In contrast to the  case of the Laplacian (\cite{KL19}), the cases $p=1$ and $p=\infty$ are excluded in the theorem. This is due the failure of the $L^1$--$L^1$ and $L^\infty$--$L^\infty$ estimates for the Riesz transform.

\subsubsection*{Eigenvalues of $-\Delta^\ast+V$}
The sharp resolvent estimates \eqref{e:shp-resol} can be  used  to  specify  the location of eigenvalues of the perturbed Lam\'e operator $-\Delta^\ast+V$ acting in $L^q(\mathbb R^d)^d$, $1< q< \infty$, for a matrix-valued potential $V\colon \mathbb R^d\to \mathcal M_{d\times d}(\mathbb C)$. It does not seem that the Birman--Schwinger principle is applicable as in \cite{Fr11, Co19} because $q\neq2$ in general.

\begin{cor} \label{c:eigen}
Let $d\ge2$, $(\frac1p, \frac1q) \in \mathcal R_1\cup\widetilde{\mathcal R}_2 \cup \widetilde{\mathcal R}_3 \cup\widetilde{\mathcal R}_3'$, and let $C$ be the constant in \eqref{e:shp-resol}. Fix a positive number $\ell>0$ (we choose $\ell\ge1$ if $1/p-1/q=2/d$). Suppose that,  for some $t\in(0,1)$,
\begin{equation}\label{e:small_potential}
	\|V\|_{L^\frac{pq}{q-p}(\mathbb R^d)^{d^2}} \le \frac{t}{C\ell d^{1-\frac1p+\frac1q}}.
\end{equation}
If $E\in\mathbb C\setminus[0,\infty)$ is an eigenvalue of $-\Delta^\ast+V$ acting in $L^q(\mathbb R^d)^d$, then $E$ must lie in $\mathbb C\setminus \mathcal Z_{p,q}(\ell)$.
\end{cor}

From the above we can deduce properties of the complex eigenvalues of the  operator $-\Delta^\ast+V$  which depend  on the potential $V$. In this regards we make a couple of remarks.

\begin{rem} If $(1/p,1/q)$ lies in the range
\[	\Big\{(x,y)\in \widetilde{\mathcal R}_2 \cup \widetilde{\mathcal R}_3 \cup \widetilde{\mathcal R}_3'\colon \frac{d-1}d < x+y < \frac{d+1}d ,~ (x,y)\neq\big(\frac12,\frac12\big)\Big\},	\]
then the region $\mathbb C\setminus\mathcal Z_{p,q}(\ell)$ is a neighborhood of $[0,\infty)$ which shrinks along the positive real line (see Figure 8(b,c,d,e) and Figure 9(e) in \cite{KL19}). Hence, for any sequence of eigenvalues $\{E_j\}$ of $-\Delta^\ast+V$ acting in $L^q(\mathbb R^d)^d$ such that $\real E_j\to\infty$, the sequence $\{\imag E_j\}$ converges to zero provided that $V$ satisfies \eqref{e:small_potential}.  Concerning analogous results for the Laplacian see \cite[p.220, Remark (1)]{Fr18} and \cite[Remark 3]{KL19}.
\end{rem}

\begin{rem}
If $1/p-1/q=2/d$ and \eqref{e:small_potential} is satisfied with some $\ell\ge1$ and $t\in(0,1)$, the region $\mathbb C\setminus\mathcal Z_{p,q}(\ell)$ is a conic region which  is contained  in the right half plane and  its apex is at the origin (see \cite[Figure 10]{KL19}). Thus, we deduce from Corollary \ref{c:eigen} that there is no eigenvalue of the operator $-\Delta^\ast+V$ acting in $L^q(\mathbb R^d)^d$ which has negative real part. We refer the reader to the recent papers \cite{CCF, Co19} for descriptions of bounds for eigenvalues of $-\Delta^\ast+V$ acting in $L^2(\mathbb R^d)^{d}$ in terms of size of $V$ measured in the Lebesgue, Morrey-Campanato, or Kerman-Sawyer spaces.
\end{rem}

\subsubsection*{Uniform Sobolev and Carleman estimates for $-\Delta^\ast$}
In view of Theorem \ref{t:resol}, the $L^p$--$L^q$ estimate for the resolvent of the Lam\'e operator displays similar behavior as that for the resolvent of the Laplacian. In particular, we have uniform estimates for the resolvent of  the Lam\'e operator when $(\frac1p,\frac1q)\in(A,A')$.  Thus, in perspective of the results in \cite{KRS87}, it is natural to expect that the following uniform Sobolev inequality holds:
\begin{equation}\label{e:usi}
	\|u\|_{L^q(\mathbb R^d)^d} \le C \|(-\Delta^\ast+M_1\nabla+M_2)u\|_{L^p(\mathbb R^d)^d},
\end{equation}
where $C$ is independent of all complex valued $(d\times d)$-matrices $M_1, M_2$. If this were true, then we could particularly deduce the following form of Carleman estimate
\begin{equation}\label{e:car}
	\|u\|_{L^q(\mathbb R^d)^d} \le C \|e^{v\cdot x}(-\Delta^\ast)e^{-v\cdot x} u\|_{L^p(\mathbb R^d)^d}
\end{equation}
with $C$ independent of $v\in\mathbb R^d\setminus\{0\}$, which would imply  the unique continuation property for the differential inequality $|\Delta^\ast u|\le |Vu|$.  As was already mentioned before, if $\lambda+2\mu=\mu$, then $-\Delta^\ast=-\mu(\Delta,\ldots, \Delta)$, and it follows from the Laplacian case (\eqref{e:car-lapl}, \eqref{e:car-iff}) that \eqref{e:car} holds if and only if $p,q$ satisfy the condition \eqref{e:car-iff}.  However, contrary to the seemingly natural expectation, the Carleman estimate \eqref{e:car} fails whenever $\lambda+2\mu\neq\mu$. Hence, the uniform Sobolev inequality \eqref{e:usi} also fails.

\begin{thm}\label{t:Carleman}
Let $d\ge 2$ and $\lambda +2\mu\neq \mu$. Then for any $1\le p,q\le \infty$, both the uniform Sobolev inequality \eqref{e:usi} and the Carleman estimate \eqref{e:car} fail.
\end{thm}

The failure contrasts with the previous result concerning the Carleman estimate \eqref{e:car-lapl} for the Laplacian (\cite{KRS87, JKL18}). While the relevant multiplier for the estimate \eqref{e:car-lapl} is $(|\xi|^2-1+2i\xi_1)^{-1}$,  the estimate \eqref{e:car}  implies $L^p$--$L^q$ boundedness of  the multiplier operator  given by the multiplier  $(\xi_1+i)\xi_2(|\xi|^2-1+2i\xi_1)^{-2}$ (see \eqref{e:car-lame} below). Compared with  $(|\xi|^2-1+2i\xi_1)^{-1}$ the function $(|\xi|^2-1+2i\xi_1)^{-2}$ exhibits more singular behavior near the sphere $\{0\}\times S^{d-2}:=\{\xi\colon |\xi|=1, \xi_1=0\}$ and this leads to failure of the estimate \eqref{e:car}.

We close the introduction with a couple of remarks.

\begin{rem} Concerning the \emph{strong} unique continuation property of Schr\"odinger operators, Jerison, and Kenig \cite{JK85, Je86} proved the following type of Carleman estimate
\begin{equation}\label{e:strong}
\|u\|_{L^q(\mathbb R^d)} \le C \||x|^{-\tau}\Delta |x|^{\tau}u\|_{L^p(\mathbb R^d)}
\end{equation}
with $p=\frac{2d}{d+2}$, $q=\frac{p}{p-1}$, and with $C$ independent of $\tau\in\mathbb R$ such that $\dist (\tau, \mathbb Z +\frac{d}{q})>0$. Later, this estimate was extended for general off-dual pairs of $p, q$ in \cite{St-append, So90, KL18}. In \cite[Proposition 2.5]{Wo93}, Wolff made a simple observation that \eqref{e:strong} implies \eqref{e:car-lapl}. By the same argument, we can easily deduce the false estimate \eqref{e:car} from the following type of Carleman estimate
\begin{equation}\label{e:strong-lame}
\|u\|_{L^q(\mathbb R^d)^d} \le C \||x|^{-\tau}\Delta^\ast |x|^{\tau}u\|_{L^p(\mathbb R^d)^d},
\end{equation}
where the constant $C>0$ is independent of  a sequence $\tau\to \infty$. Therefore, by Theorem \ref{t:Carleman}, we conclude that the estimate \eqref{e:strong-lame} is also impossible for any $p,q$.
\end{rem}

\begin{rem} A large body of literature (e.g., \cite{DR93,AITY98,W01,AM01,LW05,E06,E08}) is available regarding unique continuation for the Lam\'e system 
\begin{equation}\label{e:u-c-lame}
	\mu \Delta u +(\lambda+\mu) \nabla \div u +\nabla\lambda\,\div u +\big(\nabla u +(\nabla u)^t\big)\nabla\mu +\rho u =0.
\end{equation}
Among others, Lin, Nakamura, and Wang \cite{LNW10} proved the strong unique continuation property for \eqref{e:u-c-lame} whenever $\lambda, \mu\in C^{0,1}$, $\min\{\mu, \lambda+2\mu\}\ge \delta_0>0$, and $\rho\in L^\infty$. To the authors' knowledge, it seems that there is no result on unique continuation for the system when $\rho$ is unbounded, e.g., $\rho\in L^p$ for $p\neq \infty$. {From the typical viewpoint of applying Carleman estimates to unique continuation problem for unbounded potentials, $L^p$--$L^q$ Carleman estimate such as \eqref{e:car} need to be developed. However, Theorem \ref{t:Carleman} alludes negatively to the approach in this direction.}
\end{rem}

\subsubsection*{Acknowledgement} 
This work was supported by the KIAS Individual Grant and the National Research Foundation of Korea (NRF) with grant numbers MG073701 and NRF-2020R1F1A1A010735\linebreak2012 (Yehyun Kwon), NRF-2021R1A2B5B02001786 (Sanghyuk Lee), and NRF-2019R1F1A1061316 (Ihyeok Seo). The authors would like to thank Alberto Ruiz for informing us of the paper \cite{BFPRV} and the anonymous referee for very careful reading and various helpful comments.

\section{Resolvent estimates}\label{sec:2}
In this section, we prove Proposition \ref{p:admissible}, Theorem \ref{t:resol}, and Corollary \ref{c:eigen}. For the purpose we make use of the identity \eqref{e:helmholtz}.   For the sake of completeness we provide a Fourier-analytic proof of \eqref{e:helmholtz}.

For a function $f\colon \mathbb R^d\to \mathbb C$, we denote its Fourier transform by $\widehat f(\xi)=\mathcal F f(\xi)=\int_{\mathbb R^d}e^{-ix\cdot\xi}f(x)dx$, and  inverse Fourier transform by $\mathcal F^{-1}f(x)$.  If $f=(f_1,\ldots, f_d)$ is vector-valued, then we also write $\widehat{f}=(\widehat{f_1},\ldots, \widehat{f_d})$. For a matrix $A$, $[A]_{jk}$ denotes the $(j,k)$-component of $A$.

\begin{lem}\label{l:resol} 
Let $d\ge 2$ and $z\in \mathbb C\setminus [0,\infty)$. For every $f=(f_1,\ldots, f_d)\in\mathcal S(\mathbb R^d)^d$ and $1\le j\le d$, we have
\begin{equation}\label{e:lameresol}
	[(-\Delta^\ast-z)^{-1}f]_j 
	= (-\mu \Delta-z)^{-1}\sum_{k=1}^d (\delta_{jk}+R_jR_k)f_k - (-(\lambda+2\mu)\Delta-z)^{-1} \sum_{k=1}^d R_j R_k f_k, 
\end{equation}
where $\delta_{jk}$ is the Kronecker delta and  $R_j$ is the Riesz transform defined by $\widehat{R_jf_k}(\xi)=i\frac{\xi_j}{|\xi|}\widehat {f_k} (\xi)$.
\end{lem}
\begin{proof}
Let us formally write $u:=(-\Delta^\ast-z)^{-1}f$ and take the Fourier transform on the system $f=(-\Delta^\ast-z)u$. Then we see that
\[	(\mu |\xi|^2-z)\widehat {u_j}(\xi) +(\lambda +\mu)\bigg(\sum_{k=1}^d \xi_k\widehat{u_k}(\xi)\bigg)\xi_j =\widehat{f_j}(\xi), \quad 1\le j\le d. \]
If we regard every $d$-dimensional vector as a $(d\times 1)$-matrix, the system of equations is written as follows: 
\begin{equation}\label{e:matrix}
	\widehat f(\xi)=\big[(\mu|\xi|^2-z)I_d + (\lambda+\mu) (\xi\xi^t)\big]\widehat u(\xi)=: L_z(\xi) \widehat u(\xi).
\end{equation}
To obtain \eqref{e:lameresol} we need to invert the matrix $L_z(\xi)$. It is sufficient to show that, for  $\xi\in\mathbb R^d\setminus \{0\}$ and  $1\le j\le d$,
\begin{equation}\label{e:mult}
	 \widehat{u_j}(\xi)=\frac{\widehat{f_j}(\xi)}{\mu|\xi|^2-z}+\sum_{k=1}^d \left(\frac1{(\lambda+2\mu)|\xi|^2-z}-\frac1{\mu|\xi|^2-z}\right)\frac{\xi_j\xi_k \widehat{f_k}(\xi)}{|\xi|^2}, 
\end{equation}
which gives \eqref{e:lameresol} by the Fourier inversion formula. 

Let $\xi\neq0$. Choosing $R\in \SO(d)$ such that $R\xi=|\xi|e_1$, it is clear that
\[	\det L_z =\det \big( RL_z R^t\big) =\det \big((\mu|\xi|^2-z)I_d+(\lambda+\mu)|\xi|^2e_1e_1^t \big)
	=((\lambda+2\mu)|\xi|^2-z)(\mu|\xi|^2-z)^{d-1}.
\] 
Writing $\xi'=(\xi_2,\ldots,\xi_d)^t$, the $(1,1)$--minor $M_{11}$ of $L_z$ can be computed in a similar manner and we get
\[	M_{11}=\det\big((\mu |\xi|^2-z)I_{d-1}+(\lambda+\mu)\xi'\xi'^t\big)= ((\lambda+2\mu)|\xi|^2-(\lambda+\mu)\xi_1^2-z)(\mu|\xi|^2-z)^{d-2}. \] 
These also can be checked without difficulty  by applying elementary column (row) operations and utilizing properties of determinant.  Similarly,
\[	M_{jj}=((\lambda+2\mu)|\xi|^2-(\lambda+\mu)\xi_j^2-z)(\mu|\xi|^2-z)^{d-2}, \quad 1\le j\le d. \]
If $j\neq k$, column (or row) operations give 
\[	M_{jk}=(-1)^{j+k-1}(\lambda+\mu) \xi_j\xi_k(\mu|\xi|^2-z)^{d-2}.	\]
By Cramer's rule we see that
\[	[L_z(\xi)^{-1}]_{jk} = \frac{(-1)^{j+k}M_{kj}}{\det L_z} 
	= \frac{\delta_{jk}}{\mu|\xi|^2-z} + \left(\frac1{(\lambda+2\mu)|\xi|^2-z} - \frac1{\mu|\xi|^2-z} \right) \frac{\xi_j\xi_k}{|\xi|^2}.	\]
Since $\widehat u (\xi)= L_z(\xi)^{-1}\widehat f(\xi)$ we have \eqref{e:mult}, from which \eqref{e:lameresol} follows. 
\end{proof}

\subsection*{Proof of Theorem \ref{t:resol}}  Once we have Lemma \ref{l:resol}, the $L^p$--$L^q$ resolvent estimates for the Lam\'e operator can be deduced by making use of those for the Laplacian in \cite{KL19}. 

\subsubsection*{Upper bound in \eqref{e:shp-resol}}
First, we recall from \cite[Theorem 1.4]{KL19} the estimate
\begin{equation}\label{e:lap-res}
	\|(-\Delta-z)^{-1}\|_{p\to q}\lesssim \kappa_{p,q}(z),
\end{equation}
where $p,q$ are given as in Theorem \ref{t:resol}. By Lemma \ref{l:resol} and \eqref{e:lap-res} we then have that for every $j=1,\ldots, d$, 
\begin{align*}
&\|[(-\Delta^\ast-z)^{-1}f]_j\|_{L^q(\mathbb R^d)} \\
&\lesssim \frac{1}\mu \kappa_{p,q}\Big(\frac z\mu\Big) \|f_j\|_{L^p(\mathbb R^d)} 
	+\sum_{k=1}^d \bigg(\frac{1}{\lambda+2\mu}\kappa_{p,q}\Big(\frac z{\lambda+2\mu}\Big) +\frac{1}{\mu}\kappa_{p,q}\Big(\frac z\mu\Big) \bigg)\|R_j R_k f_k\|_{L^p(\mathbb R^d)}\\
&\le \bigg[ \mu^{-\frac d2(\frac1p-\frac1q)} + \Big((\lambda+2\mu)^{-\frac d2(\frac1p-\frac1q)}+\mu^{-\frac d2(\frac1p-\frac1q)}\Big) \tan^2\Big(\frac{\pi}{2\min\{p, p'\}}\Big) \bigg] \kappa_{p,q}(z) \|f\|_{L^p(\mathbb R^d)^d}
\end{align*}
since $\kappa_{p,q}(z/\mu)=\mu^{1-\frac d2(\frac1p-\frac1q)}\kappa_{p,q}(z)$, and the Riesz transforms are bounded on $L^p(\mathbb R^d)$, $1<p<\infty$, with norm $\|R_j\|_{p\to p}=\tan(\frac{\pi}{2\min\{p, p'\}})$. (See \cite[Theorem 3]{BW95}). This proves the upper bound in \eqref{e:shp-resol}.

\subsubsection*{Proof of \eqref{e:rest-wk-type} and \eqref{e:wk-type}}
For the restricted weak type bound \eqref{e:rest-wk-type} we argue similarly using the restricted weak type $(p,q)$ bound for the Laplacian resolvent (\cite[Theorem 1.4]{KL19}) and the $L^{p,1}$--$L^{p,1}$ estimate of the Riesz transforms (see \cite[Theorem 1.1]{Os14}). Indeed, 
\begin{align*}
&\|[(-\Delta^\ast-z)^{-1}f]_j\|_{L^{q,\infty}(\mathbb R^d)} \lesssim \frac1\mu\|f_j\|_{L^{p,1} (\mathbb R^d)} 
	+ \bigg(\frac1{\lambda+2\mu}+\frac1\mu \bigg) \sum_{k=1}^d \|R_j R_k f_k\|_{L^{p,1}(\mathbb R^d)} \lesssim \|f\|_{L^{p,1}(\mathbb R^d)^d}
\end{align*}
The weak type bound \eqref{e:wk-type} can also be shown in the similar way using the weak type $(p,q)$ bound for $(-\Delta-z)^{-1}$ in \cite[Theorem 1.4]{KL19}.

\subsubsection*{Lower bound in \eqref{e:shp-resol}}  On the other hand, the lower bound in \eqref{e:shp-resol} can be obtained  by considering functions whose Fourier transform is  supported near  the  sphere $S^{d-1}$  as it  was done in \cite[Lemma 5.1]{KL19}. However, as each component of the multiplier is not rotationally symmetric, we cannot directly use the well-known asymptotic of the Bessel function anymore as in \cite{KL19}. In order to get around this difficulty, we instead apply the stationary phase method to obtain the asymptotic of oscillatory integrals.

We now show the lower bound in \eqref{e:shp-resol}. It is enough to show, taking $f=(f_1, 0, \ldots, 0)$ and $j=1$ in \eqref{e:lameresol}, that
\[	\| (-\mu\Delta-z)^{-1}(1+R_1^2)- (-(\lambda+2\mu)\Delta-z)^{-1}R_1^2 \|_{p\to q} \gtrsim \kappa_{p,q}(z). \] 
Multiplying this by $\mu$ and then replace $z$ with $\mu z$, this is equivalent to 
\[	\Big\| (-\Delta-z)^{-1}(1+R_1^2) - \big(-\frac{\lambda+2\mu}{\mu}\Delta-z\big)^{-1}R_1^2 \Big\|_{p\to q} \gtrsim \mu\kappa_{p,q}(\mu z) \approx \kappa_{p,q}(z),	\]
which is again identical with 
\[	\Big\| \big(-\Delta-\frac z{|z|} \big)^{-1}(1+R_1^2) - \big(-\frac{\lambda+2\mu}{\mu}\Delta- \frac z{|z|} \big)^{-1}R_1^2 \Big\|_{p\to q} \gtrsim \dist \big(\frac z{|z|},[0,\infty)\big)^{-\gamma_{p,q}}	\]
since $\|(-\Delta-z)^{-1}\|_{p\to q} = |z|^{-1+\frac d2(\frac1p-\frac1q)} \|(-\Delta-\frac z{|z|})^{-1}\|_{p\to q}$ and the multiplier of the Riesz transform $R_1$ is of homogeneous of degree zero. (The identity can be readily checked by the Fourier inversion formula and scaling; $x\to |z|^{-\frac12}x$ for the space variable and $\xi\to |z|^\frac12\xi$ for the frequency variable.) Hence, recalling the definition of $\gamma_{p,q}$ it is enough to show that, for $0<\delta\ll 1$,
\begin{equation}\label{e:low}
	\begin{aligned}
		&\Big\| \big(-\Delta-(1+i\delta)\big)^{-1} (1+R_1^2) - \big(-\rho \Delta-(1+i\delta) \big)^{-1} R_1^2 \Big\|_{p\to q} \\
		&\qquad \qquad \gtrsim \max \big\{1, \delta^{-1+\frac{d+1}2(\frac1p-\frac1q)}, \delta^{\frac{d-1}2-\frac dq}, \delta^{\frac dp-\frac{d+1}2} \big\}	
	\end{aligned}
\end{equation}
under the assumption $\rho:=\frac{\lambda+2\mu}\mu\neq1$.

In the above \eqref{e:low} the first lower bound is clear because the operator is nontrivial. And the last lower bound follows from the third if we replace $1+i\delta$ with $1-i\delta$, by the following general principle of duality (for every linear operator $T$ and its adjoint $T^\ast$)
\begin{equation}\label{e:duality}	
	\|T^\ast\|_{p\to q} =\|T\|_{q'\to p'}.
\end{equation}
Hence it suffices to show the second and third.  If we denote the multiplier in \eqref{e:low} by 
\[	a_\delta (\xi) = \frac{1}{|\xi|^2-(1+i\delta)}\frac{|\xi'|^2}{|\xi|^2} + \frac{1}{\rho|\xi|^2-(1+i\delta)}\frac{\xi_1^2}{|\xi|^2}, \quad \xi'=(\xi_2,\ldots, \xi_d)\in \mathbb R^{d-1},  	\]
and use the standard notation $a_\delta (D)f=\mathcal F^{-1}(a_\delta \widehat f\, )$ (which is of course equal to the operator in \eqref{e:low}), then it is easy to see that $\|a_\delta(D)\|_{p\to q}=\|\overline{a_\delta}(D)\|_{p\to q}$ since $a_\delta(\xi)$ is invariant under the reflection $\xi\to -\xi$. Consequently, we have 
\[	2\|a_\delta (D)\|_{p\to q}\ge \|a_\delta (D)\pm \overline{a_\delta} (D)\|_{p\to q}. 	\] 
Thus, to show \eqref{e:low}  it is enough to consider  the imaginary part  $\mathcal I_\delta:=\imag a_{\delta}$, that is,  
\[	\mathcal I_\delta(\xi)=\frac\delta{(|\xi|^2-1)^2+\delta^2} \frac{|\xi'|^2}{|\xi|^2} +\frac\delta{(\rho|\xi|^2-1)^2+\delta^2}\frac{\xi_1^2}{|\xi|^2}=: M_\delta(\xi)+m_\delta(\xi). \]

Thus, for the lower bound \eqref{e:low}, it suffices to show the following proposition, which  completes  the  proof of Theorem \ref{t:resol}. Let us set $D=-i\nabla$ and, for any bounded measurable function $m$ on $\mathbb R^d$, let $m(D)$ denote the Fourier multiplier with multiplier (symbol) $m$, i.e.,  $m(D)f=\mathcal F^{-1}(m\widehat f \,)$.
\begin{prop}\label{p:optimal}
	Let $d\ge 2$, $1\le p,q\le \infty$ and let $0<\delta\ll 1$. Then 
\begin{align}
\label{e:low-im}
	\|\mathcal I_\delta(D)\|_{p\to q} &\gtrsim \delta^{-1+\frac{d+1}2(\frac1p-\frac1q)},\\
\label{e:low-im2}
	\|\mathcal I_\delta(D)\|_{p\to q} &\gtrsim \delta^{\frac{d-1}2-\frac dq}.
\end{align}	
\end{prop}
Proving Proposition \ref{p:optimal} is a messy affair and we shall therefore hold off doing so until the next section. We close this section with  the proofs of Proposition \ref{p:admissible} and Corollary \ref{c:eigen}.

\begin{proof}[{\bf Proof of Proposition \ref{p:admissible}}]
The proof is similar to that of the necessary part of \cite[Proposition 1.1]{KL19}. It is well-known that the condition $p\le q$ is necessary for the resolvent estimate since every Fourier multiplier on $\mathbb R^d$ is translation invariant (\cite[Theorem 1.1]{Ho60}). For the other necessary conditions, let us first assume that $d\ge 3$. In view of \eqref{e:low}, it is enough to show that, for every $\rho>0$ with $\rho\neq 1$ and every $z\in S^1\setminus\{1\}$,
\begin{equation}\label{e:adm1}
	\|(-\Delta-z)^{-1}(1+R_1^2)- (-\rho \Delta-z )^{-1} R_1^2 \|_{p\to q} <\infty
\end{equation}
holds only if 
\begin{equation}\label{e:adm2}
	1/p-1/q\le 2/d, \quad (p,q)\neq (d/2,\infty), \quad \text{and} \quad (p,q)\neq (1, d/(d-2)).	
\end{equation}
By duality \eqref{e:duality}, we need only to consider the first and second conditions. Let $\phi\in C^\infty_c(\mathbb R)$ be the standard Littlewood--Paley bump function supported on the interval $[1/2,2]$ satisfying 
\[	\sum_{k\ge 0} \phi(2^{-k}t)=1, \quad t\ge 1, \]
and let $\phi_0(t):=1-\sum_{k\ge 0}\phi(2^kt)$.  Let us also define projection operators $P_k$ and $\mathcal P_0$ by $\widehat{P_kf}(\xi)=\phi(2^{-k}|\xi|)\widehat f(\xi)$ and $\widehat{\mathcal P_0f}(\xi)=\phi_0(|\xi|)\widehat f(\xi)$, respectively. Testing the bound \eqref{e:adm1} with $P_kf$ and scaling $\xi\to 2^{k}\xi$ give the inequality
\[	\bigg\|\int_{\mathbb R^d} e^{ix\cdot\xi} \Big(\frac1{|\xi|^2-2^{-2k}z}\frac{|\xi'|^2}{|\xi|^2} + \frac1{\rho|\xi|^2-2^{-2k}z}\frac{\xi_1^2}{|\xi|^2} \Big) \phi(|\xi|) \widehat g(\xi) d\xi \bigg\|_{L_x^q(\mathbb R^d)} \lesssim 2^{k(2+\frac dq-\frac dp)}\|g\|_{L^p(\mathbb R^d)}	\]
for all $g\in L^p(\mathbb R^d)$. If we consider the limit $k\to \infty$, the left side converges to a nonzero value provided that $g\neq0$ is suitably chosen, while the right side converges to zero whenever $1/p-1/q>2/d$. This is contradiction to the assumption \eqref{e:adm1}, which shows that the first condition of \eqref{e:adm2} is necessary for \eqref{e:adm1}.

Now, let us assume that \eqref{e:adm1} is true for $(p,q)=(d/2,\infty)$. Denoting the multiplier of \eqref{e:adm1} by $m_z(\xi)$ we then have
\[	\| |D|^{-2}(1-\phi_0(|D|))f \|_{L^\infty(\mathbb R^d)} \lesssim \|m_z(D)^{-1} |D|^{-2}(1-\phi_0(|D|))f \|_{L^{\frac d2}(\mathbb R^d)}.	\]
Since the multiplier $m_z(\xi)^{-1}|\xi|^{-2}(1-\phi_0(|\xi|))$ satisfies the Mikhlin's condition, it follows from Mikhlin's multiplier theorem that 
\[	\| |D|^{-2}(1-\phi_0(|D|))f \|_{L^\infty(\mathbb R^d)} \lesssim \|f\|_{L^\frac d2(\mathbb R^d)},	\]
which, by scaling, implies
\[	\| |D|^{-2}(1-\phi_0(\epsilon |D|))f \|_{L^\infty(\mathbb R^d)} \lesssim \|f\|_{L^\frac d2(\mathbb R^d)}, \quad \forall \epsilon>0.	\]
By Fatou's lemma this estimate in turn implies the following inequality
\[	\|(-\Delta)^{-1}f\|_{L^\infty(\mathbb R^d)} \lesssim \|f\|_{L^\frac d2(\mathbb R^d)},\]
which is false. Thus, we see that the second condition in \eqref{e:adm2} is necessary for \eqref{e:adm1}.

Finally, we assume $d=2$ and show that the estimate \eqref{e:adm1} fails when $(p,q)=(1,\infty)$. Note that the multiplier operator in \eqref{e:adm1} is written
\[	T_zh(x):=\frac1{(2\pi)^2} \int_{\mathbb R^2} e^{ix\cdot \xi} \Big(\frac1{|\xi|^2-z}\frac{\xi_2^2}{|\xi|^2} +\frac1{\rho|\xi|^2-z}\frac{\xi_1^2}{|\xi|^2}\Big) \widehat h(\xi) d\xi. \] 
For every $\epsilon>0$, let us define $h_\epsilon \in \mathcal S(\mathbb R^2)$ by 
\begin{equation}\label{e:def_h}
	\widehat{h_\epsilon}(\xi) = \phi_0(\epsilon^2\xi)-\phi_0(\epsilon\xi).
\end{equation}
By the definition of $\phi_0$ it is clear that $0\le \widehat{h_\epsilon} \le1$,
\begin{gather}
	\supp \widehat{h_\epsilon} \subset \{\xi\in \mathbb R^2\colon 1/2\epsilon \le |\xi| \le 1/\epsilon^2\}, \label{e:supp} \\ 
	\widehat {h_\epsilon} (\xi) = 1 \quad \text{if} \quad 1/\epsilon\le |\xi| \le 1/2\epsilon^{2},  \label{e:supp1}
\end{gather}
and
\begin{equation} \label{e:unif_norm}
	\|h_\epsilon\|_{L^1(\mathbb R^2)}\le 2 \|\mathcal F^{-1}\phi_0\|_{L^1(\mathbb R^2)}\lesssim 1.
\end{equation} 
If the estimate \eqref{e:adm1} were true for $(p,q)=(1,\infty)$, then it would also be true that 
\begin{equation}\label{e:fake}
	|\real T_zh_\epsilon(0)|\le |T_zh_\epsilon(0)| \le \|T_zh_\epsilon\|_{L^\infty(\mathbb R^2)} \lesssim \|h_\epsilon\|_{L^1(\mathbb R^2)}\lesssim 1
\end{equation}
with all the inequalities uniform in $\epsilon>0$. Let us now examine $|\real T_zh_\epsilon(0)|$. If we write $z=a+ib\in S^1\setminus\{1\}$, then 
\begin{equation}\label{e:re-tz}
	\real T_zh_\epsilon(0)
	= \frac{1}{(2\pi)^2}\int_{\mathbb R^2} \bigg(\frac{|\xi|^2-a}{(|\xi|^2-a)^2+b^2}\frac{\xi_2^2}{|\xi|^2} +\frac{\rho|\xi|^2-a}{(\rho|\xi|^2-a)^2+b^2} \frac{\xi_1^2}{|\xi|^2}\bigg) \widehat{h_\epsilon}(\xi) d\xi .
\end{equation}
Since $a^2+b^2=1$ it is easy to see that, for $|\xi|\ge 2$ satisfying $|\xi_2|\ge |\xi_1|$,
\[	\frac{|\xi|^2-a}{(|\xi|^2-a)^2+b^2} \frac{\xi_2^2}{|\xi|^2} \ge \frac{|\xi|^2-1}{2(|\xi|^2+1)^2} \gtrsim \frac{1}{|\xi|^2}.	\]
Similarly, if $|\xi|\ge \frac 2{\sqrt{\rho}}$ and $|\xi_2|\le |\xi_1|$, then
\[	\frac{\rho|\xi|^2-a}{(\rho|\xi|^2-a)^2+b^2} \frac{\xi_1^2}{|\xi|^2} \ge \frac{\rho|\xi|^2-1}{2(\rho|\xi|^2+1)^2} \gtrsim \frac 1{\rho|\xi|^2}.	\]
For every $\epsilon\le \frac12 \min\{1,\sqrt{\rho}\}$, it follows from \eqref{e:supp1} and \eqref{e:re-tz} that 
\begin{align*}
	|\real T_zh_\epsilon(0)|
	&\gtrsim \int_{1/\epsilon \le |\xi|\le 1/2\epsilon^2} |\xi|^{-2} \mathbf{1}_{\{\xi\colon |\xi_2|\ge|\xi_1|\}} (\xi)  d\xi
		+ (\rho|\xi|)^{-2} \mathbf{1}_{\{\xi\colon |\xi_2|\le|\xi_1|\}} (\xi)  d\xi \\
	&\gtrsim \int_{1/\epsilon \le |\xi|\lesssim 1/\epsilon^2} |\xi|^{-2}d\xi \approx \log \frac1\epsilon,
\end{align*}
which diverges to infinity as $\epsilon \to 0$. This is contradiction to \eqref{e:fake}. Thus, the estimate \eqref{e:adm1} fails with $(p,q)=(1,\infty)$ when $d=2$. 
\end{proof}

\begin{proof}[{\bf Proof of Corollary \ref{c:eigen}}]
Since $E$ is an $L^q$-eigenvalue of $-\Delta^\ast+V$, there exists a nonzero $u\in L^q(\mathbb R^d)^d$ such that $(-\Delta^\ast+V)u=Eu$. Assume that $E$ lies in $\mathcal Z_{p,q}(\ell)$. Since $q\ge p$ H\"older's and Minkowski's inequalities give
\begin{align*}
	\|Vu\|_{L^p(\mathbb R^d)^d}
	&\le \bigg (\sum_{j=1}^d \Big( \sum_{k=1}^d \|V_{jk}\|_{L^\frac{pq}{q-p}(\mathbb R^d)}^{q'} \Big)^{\frac{p}{q'}} \bigg)^\frac1p \|u\|_{L^q(\mathbb R^d)^d}
	=: \|V_{jk}\|_{l_j^p l_k^{q'}L^\frac{pq}{q-p}} \|u\|_{L^q(\mathbb R^d)^d}.
\end{align*}
Now we recall the inequality 
\begin{equation}\label{e:summ}
	\Big(\sum_{i=1}^d a_i \Big)^\theta\le d^{\theta-1} \sum_{i=1}^d a_i^\theta
\end{equation} 
which holds for every $\theta\ge1$ and $a_i\ge 0$. Applying this inequality with ($i=j$ and) 
\[	a_j=\Big( \sum_{k=1}^d \|V_{jk}\|_{L^\frac{pq}{q-p}(\mathbb R^d)}^{q'} \Big)^{\frac{p}{q'}} \quad \text{and} \quad \theta= \frac q{q-p}	\]
we see that
\[	\|V_{jk}\|_{l_j^p l_k^{q'}L^\frac{pq}{q-p}} 
	\le \bigg( d^{\frac{q}{q-p}-1} \sum_{j=1}^d \Big( \sum_{k=1}^d \|V_{jk}\|_{L^\frac{pq}{q-p}(\mathbb R^d)}^{q'} \Big)^{\frac{p(q-1)}{q-p}} \bigg)^\frac{q-p}{pq}. 	\]
Similarly, applying \eqref{e:summ} with ($i=k$ and) 
\[	a_k= \|V_{jk}\|_{L^\frac{pq}{q-p}(\mathbb R^d)}^{q'} \quad \text{and} \quad \theta= \frac{p(q-1)}{q-p}	\]
we have
\[	\|V_{jk}\|_{l_j^p l_k^{q'}L^\frac{pq}{q-p}} 
	\le \bigg( d^{\frac{q}{q-p}+\frac{p(q-1)}{q-p} -2} \sum_{j=1}^d  \sum_{k=1}^d \|V_{jk}\|_{L^\frac{pq}{q-p}(\mathbb R^d)}^\frac{pq}{q-p} \bigg)^\frac{q-p}{pq}
	= d^{1-\frac1p+\frac1q}\|V\|_{L^\frac{pq}{q-p}(\mathbb R^d)^{d^2}}. 	\]
From the resolvent estimate \eqref{e:shp-resol} and the assumptions $E\in \mathcal Z_{p,q}(\ell)$, $(-\Delta^\ast+V)u=Eu$, and \eqref{e:small_potential}, it follows that
\begin{align*}
	\|u\|_{L^q(\mathbb R^d)^d} 
	&\le C\kappa_{p,q}(E) \|(-\Delta^\ast-E)u\|_{L^p(\mathbb R^d)^d} \\
	&\le C\ell \big(\|(-\Delta^\ast+V-E)u\|_{L^p(\mathbb R^d)^d} +\|Vu\|_{L^p(\mathbb R^d)^d}\big) \\
	&\le C\ell d^{1-\frac1p+\frac1q}\|V\|_{L^\frac{pq}{q-p}(\mathbb R^d)^{d^2}} \|u\|_{L^q(\mathbb R^d)^d} \\
	&\le t\|u\|_{L^q(\mathbb R^d)^d},
\end{align*}
which forces $u$ to be identically zero since $t\in (0,1)$. This  contradicts that $u$ is nonzero. Therefore, $E$ must lie in  $\mathbb C\setminus\mathcal Z_{p,q}(\ell)$.
\end{proof}

\section{Proof of Proposition \ref{p:optimal}}
In this section, we prove Proposition \ref{p:optimal}.

\begin{proof}[Proof of \eqref{e:low-im}] Let us choose $\phi, \psi \in C_c^\infty(\mathbb R)$ such that $0\le \phi, \psi\le 1$, $\supp\phi\subset[-1,1]$, $\phi=1$ on $[-1/2,1/2]$, $\supp\psi\subset[1/4,1]$, and $\psi=1$ on $[1/2,3/4]$. Now we define $h_\delta\in\mathcal S(\mathbb R^d)$ by 
\[	\widehat {h_\delta}(\xi)=\psi\Big(\frac{\xi_d-1}{\delta}\Big)\prod_{j=1}^{d-1}\phi\Big(\frac{\xi_j}{\sqrt\delta} \Big). \]
It is obvious that 
\begin{equation}\label{e:suppsize}
	|\supp \widehat{h_\delta}|\approx |\{\xi\in \mathbb R^d\colon \widehat{h_\delta}(\xi)=1 \} |\approx \delta^{1+\frac{d-1}2}.
\end{equation}
By the Fourier inversion formula it is also clear that 
\[	h_\delta(x)=e^{ix_d} \delta \big(\mathcal F^{-1}\psi\big)(\delta x_d) \prod_{j=1}^{d-1} \sqrt\delta \big(\mathcal F^{-1}\phi\big) (\sqrt\delta x_j),	\] 
from which it follows that 
\begin{equation}\label{e:lpnormh}
	\|h_\delta\|_{L^p(\mathbb R^d)}\approx\delta^{\frac{d+1}2 -\frac{d+1}{2p}}. 
\end{equation}
If $\xi\in \supp \widehat{h_\delta}$ then, writing $\xi=(\eta, \tau)\in \mathbb R^{d-1}\times\mathbb R$, we see that
\begin{gather*}
	|\xi|^2-1=|\eta|^2+(\tau+1)(\tau-1)\approx \delta,  \\
	|\xi|^2=1+|\eta|^2+(\tau+1)(\tau-1)\approx 1+\delta \approx 1
\end{gather*} 
since $\delta$ is small. Similarly, it is easy to check that $|\xi'|^2\approx 1$ for $\xi\in \supp \widehat{h_\delta}$.  Hence,
\[	\delta^{-1} \ge M_\delta(\xi) := \frac\delta{(|\xi|^2-1)^2+\delta^2} \frac{|\xi'|^2}{|\xi|^2} \approx \frac{\delta}{\delta^2+\delta^2} \approx \delta^{-1}, \quad \xi\in \supp \widehat {h_\delta}. 	\]
On the other hand, since $\sqrt{\rho}\neq 1$, 
\[	\big|\rho|\xi|^2-1\big| = \big|\rho|\eta|^2 +(\sqrt{\rho}\tau+1)(\sqrt{\rho}(\tau-1)+\sqrt{\rho}-1)\big| \gtrsim 1, \quad \xi \in \supp \widehat{h_\delta}.	\]
Hence $0\le m_\delta(\xi) \lesssim \delta$ on $\supp \widehat{h_\delta}$, and we have  
\begin{equation}\label{e:m-size}
	\mathcal I_\delta(\xi) =M_\delta(\xi)+m_\delta(\xi) \gtrsim  \delta^{-1}, \quad \xi\in  \supp \widehat {h_\delta}.
\end{equation}
Let us set 
\[	A_{\delta}:=\{(y,t)\in\mathbb R^{d-1}\times \mathbb R\colon |t|\le (100\delta)^{-1},\, |y_j|\le (100d\sqrt{\delta})^{-1}, \, 1\le j\le d-1 \}. \]
It is easy to see that if $x=(y,t)\in A_\delta$ and $\xi=(\eta ,\tau)\in \supp \widehat{h_\delta}$, then
\[	|x \cdot (\xi - e_d )| 
	\le |y| |\eta | + |t | | \tau-1| 
	\le \frac{\sqrt{d-1}}{100d\sqrt{\delta}}  \cdot \sqrt{(d-1)\delta} + \frac{1}{100\delta} \cdot \delta \le \frac1{50},	\]
and it follows that 
\begin{equation} \label{e:cos-1}
	\cos(x\cdot(\xi-e_d)) \ge 1-\frac {|x \cdot (\xi - e_d )|^2}2 \gtrsim 1.
\end{equation}
Hence, by the estimates \eqref{e:suppsize}, \eqref{e:m-size}, and \eqref{e:cos-1}, we have 
\begin{align*}
	|\mathcal I_\delta(D)h_\delta(x)|
	&= \bigg| \frac1{(2\pi)^d }\int_{\mathbb R^d} e^{ix\cdot(\xi-e_d)} \mathcal I_\delta(\xi)\widehat{h_\delta}(\xi) d\xi \bigg| \\
	&\ge  \frac1{(2\pi)^d } \int_{\{\xi \in \mathbb R^d \colon \widehat{h_\delta}(\xi)=1 \}} \cos(x\cdot(\xi-e_d)) \mathcal I_\delta(\xi) d\xi 
	\gtrsim  \delta^{\frac{d-1}{2}}.
\end{align*}
Since $|A_\delta| \approx \delta^{-\frac{d+1}2}$ this estimate gives  $\|\mathcal I_\delta(D)h_\delta\|_{L^q(A_\delta)} \gtrsim \delta^{\frac{d-1}2 -\frac{d+1}{2q}}$. Therefore, from \eqref{e:lpnormh} we conclude that 
\[	\|\mathcal I_\delta(D)\|_{p\to q} 
	:=\sup_{f\neq 0} \frac{\|\mathcal I_\delta (D)f\|_{L^q(\mathbb R^d)}}{\|f\|_{L^p(\mathbb R^d)}}
	\ge\frac{\|\mathcal I_\delta (D)h_\delta\|_{L^q(A_\delta)}}{\|h_\delta\|_{L^p(\mathbb R^d)}}
	\gtrsim \frac{\delta^{\frac{d-1}2 -\frac{d+1}{2q}}}{\delta^{\frac{d+1}2 -\frac{d+1}{2p}}}=\delta^{-1+\frac{d+1}2(\frac1p-\frac1q)},
\]
which shows the bound \eqref{e:low-im}.
\end{proof}

From now on, we write $\xi=(\xi_1,\xi')=(\xi_1,\xi_2, \bar\xi)\in\mathbb R\times\mathbb R\times \mathbb R^{d-2}$.  Also, we sometimes write $\tau=\xi_1$. 
\begin{proof}[Proof of \eqref{e:low-im2}]
By scaling, we note that $\| \mathcal I_\delta(D)\|_{p\to q}= \rho^{-\frac d2(\frac1p-\frac1q)}\|\widetilde{\mathcal I}_\delta(D)\|_{p\to q}$, 
where 
\[	\widetilde{\mathcal I}_\delta(\xi):= \mathcal{I}_\delta(\rho^{-\frac12}\xi) = \frac\delta{(\rho^{-1}|\xi|^2-1)^2+\delta^2}\frac{|\xi'|^2}{|\xi|^2} + \frac{\delta}{(|\xi|^2-1)^2+\delta^2}\frac{\xi_1^2}{|\xi|^2}.	\]  
Thus, in order to prove \eqref{e:low-im}, it is harmless to assume that $\mathcal I_\delta=\widetilde{\mathcal I}_\delta$, 
\[	M_\delta(\xi)=\frac\delta{(\rho^{-1}|\xi|^2-1)^2+\delta^2}\frac{|\xi'|^2}{|\xi|^2}, \quad \text{and} \quad  m_\delta(\xi)=\frac{\delta}{(|\xi|^2-1)^2+\delta^2}\frac{\xi_1^2}{|\xi|^2}.	\]
If we put $\psi(\xi'):= 1-\sqrt{1-|\xi'|^2}$ and apply change of variables via diffeomorphism $\xi\to (\xi_1-1+\psi(\xi'), \xi')$, then we have
\begin{align*}
	m_\delta(D)f(x) = \frac{e^{-ix_1}}{(2\pi)^d} \int_{\mathbb R} e^{i x_1\xi_1} & \int_{\mathbb R^{d-1}} e^{i(x'\cdot\xi' +x_1\psi(\xi'))} \frac\delta{\xi_1^2(\xi_1+2\psi(\xi')-2)^2+\delta^2} \\
	&\qquad \times \frac{(\xi_1+\psi(\xi')-1)^2}{(\xi_1+\psi(\xi')-1)^2+|\xi'|^2} \widehat f (\xi_1+\psi(\xi')-1, \xi') d\xi' d\xi_1.
\end{align*}
We then choose a function $f\in \mathcal S(\mathbb R^d)$ so that
\begin{equation}\label{e:choice-max}
	\frac{(\xi_1+\psi(\xi')-1)^2}{(\xi_1+\psi(\xi')-1)^2+|\xi'|^2} \widehat f (\xi_1+\psi(\xi')-1, \xi') = \chi(\xi') \varphi(\xi_1), 
\end{equation}
where $\chi\in C_c^\infty(B_{d-1}(0,\frac1{10}))$ such that  $0\le \chi\le 1$ and $\chi(\xi')=1$ if $|\xi'|\le \frac1{20}$, and $\varphi\in C_c^\infty((-2\epsilon_\circ, 2\epsilon_\circ))$ satisfying  $0\le\varphi\le1$ and $\varphi(t)=1$ if $|t|\le \epsilon_\circ$. Here, $\epsilon_\circ>0$ is a fixed small number to be determined depending on $\rho$ (see Lemma \ref{l:symbol} below and its proof).

With the choice of $f$, $m_\delta(D)f$ is now written as the following favorable form:
\[	m_\delta(D)f(x) = \frac{e^{-ix_1}}{(2\pi)^d} \int_{-2\epsilon_\circ}^{2\epsilon_\circ} e^{i x_1\tau}  I_\delta(x;\tau) \varphi(\tau) d\tau,	\]
where we set 
\[	I_\delta(x;\tau):=\int e^{i(x'\cdot\xi' +x_1\psi(\xi'))} a_\delta(\tau,\xi') \chi(\xi') d\xi' , \quad a_\delta(\tau, \xi') := \frac\delta{\tau^2(\tau+2\psi(\xi')-2)^2+\delta^2}.	\] 
By Lemma \ref{l:asymp} below and the triangle inequality, if $x_1\ge 1/2$ and $2^5|x'|\le x_1$, we have
\begin{equation}\label{e:main-term}
	\begin{aligned}
	|m_\delta(D)f(x)|
		&\approx \bigg| \int_{-2\epsilon_\circ}^{2\epsilon_\circ} e^{ix_1\tau} \bigg[ e^{i(x_1-|x|)}\sum_{j=0}^{N-1}x_{1}^{-\frac{d-1}2-j}\mathcal D_ja_\delta(\tau,\xi')\vert_{\xi'=-\frac{x'}{|x|}} +\mathcal E_{\delta, N}(x;\tau) \bigg] \varphi(\tau) d\tau\bigg| \\
		&\ge x_1^{-\frac{d-1}2}\bigg| \int e^{ix_1\tau} a_\delta\Big(\tau,-\frac{x'}{|x|}\Big) \varphi(\tau) d\tau \bigg| \\
		&\quad - \sum_{j=1}^{N-1} x_1^{-\frac{d-1}2-j}\int \Big|\mathcal D_ja_\delta(\tau, \xi')\vert_{\xi'=-\frac{x'}{|x|}}\varphi(\tau)\Big| d\tau -\int \Big|\mathcal E_{\delta, N}(x;\tau) \varphi(\tau) \Big| d\tau \\
		&=:Q_0(x)-\sum_{j=1}^{N-1}Q_j(x)-R_N(x).
	\end{aligned}
\end{equation}
Now, for a large number $\nu>0$ to be chosen shortly, let us define the set 
\[	B_\delta:= \big\{x\in\mathbb R^d\colon (20\nu\delta)^{-1} \le x_1\le (10\nu\delta)^{-1}, \, x_1\ge  2^5|x'| \big\}.	\]
Then we break the integral in the term $Q_0(x)$ as
\[	\int_{-\nu\delta}^{\nu\delta} e^{ix_1\tau} a_\delta\Big(\tau,-\frac{x'}{|x|}\Big) \varphi(\tau) d\tau + \bigg(\int_{-2\epsilon_\circ}^{-\nu\delta} +\int_{\nu\delta}^{2\epsilon_\circ} \bigg) e^{ix_1\tau} a_\delta\Big(\tau,-\frac{x'}{|x|}\Big) \varphi(\tau) d\tau =:\widetilde Q_{0}(x) +\widetilde R_{0}(x).	\]
If $x\in B_\delta$, we have $\frac{|x'|}{|x|}\le \frac{1}{20}$. Also, since $|\tau|\le 2\epsilon_\circ$ it is easy to see that $1-2\epsilon_\circ \le |\tau+2\psi(-\frac{x'}{|x|})-2|\le 2+2\epsilon_\circ$. Hence, if $\nu\delta\le \epsilon_\circ$, for any $x\in B_{\delta}$,
\[	|\widetilde Q_0(x)|
	\ge \int_{-\nu\delta}^{\nu\delta} \cos(x_1\tau) a_{\delta}\Big(\tau,-\frac{x'}{|x|}\Big) d\tau
	\gtrsim \int_{-\nu\delta}^{\nu\delta} \frac{\delta}{(2+2\epsilon_\circ)^2\tau^2+\delta^2}d\tau 
	\gtrsim \int_{-\nu}^\nu \frac{1}{\tau^2+1}d\tau.
\]
Similarly, for any $x\in B_{\delta}$, it is clear that 
\[	|\widetilde R_0(x)| 
	\lesssim \int_{\nu\delta}^{2\epsilon_\circ} \frac{\delta}{(1-2\epsilon_\circ)^2\tau^2+\delta^2}d\tau 
	\lesssim  \int_\nu^\infty \frac{1}{\tau^2+1}d\tau.
\]
If we choose $\nu$ large enough, we have 
\begin{equation}\label{e:q1}
	Q_0(x) \ge x_1^{-\frac{d-1}2} \big(|\widetilde Q_0(x)|-|\widetilde R_0(x)| \big) \gtrsim \delta^\frac{d-1}2, \quad \forall x\in B_\delta.
\end{equation}

For $1\le j\le N-1$, by the estimate \eqref{e:symbol} below, we see that
\begin{equation}\label{e:q2}
	Q_j(x)\lesssim \delta^{\frac{d-1}2+j} \int\frac{\delta}{(1-2\epsilon_\circ)^2\tau^2+\delta^2} d\tau \lesssim \delta^{\frac{d-1}2+j}, \quad \forall x\in B_\delta.
\end{equation}
Now we utilize the estimates \eqref{e:asymp-err} and \eqref{e:symbol} to  obtain
\begin{equation}\label{e:q3}
	R_N(x)\lesssim \delta^N \int_{-2\epsilon_\circ}^{2\epsilon_\circ} \sum_{|\alpha|\le 2N} \sup_{(\tau,\xi')} |\partial_{\xi'}^\alpha a_\delta(\tau, \xi')|d\tau \lesssim \delta^{N-1}, \quad \forall x\in B_\delta.
\end{equation}
Choosing $N$ large enough and combining all together the estimates \eqref{e:main-term}, \eqref{e:q1}, \eqref{e:q2}, and \eqref{e:q3}, we conclude that 
\begin{equation}\label{e:est-main-part}
	\|m_\delta(D)f\|_{L^q(B_\delta)} \gtrsim \delta^{\frac{d-1}2}|B_\delta|^{\frac1q} \approx \delta^{\frac{d-1}2-\frac dq}.
\end{equation}

On the other hand, the same change of variables as before in the frequency domain gives 
\[	M_\delta(D)f(x) = \frac{e^{-ix_1}}{(2\pi)^d} \int_{-2\epsilon_\circ}^{2\epsilon_\circ} e^{i x_1\tau}  J_\delta(x;\tau) \varphi(\tau) d\tau,
	\quad J_\delta(x;\tau):=\int e^{i(x'\cdot\xi' +x_1\psi(\xi'))} b_\delta(\tau,\xi') \chi(\xi') d\xi',
\]
where
\[	b_\delta(\tau,\xi') := \frac{\rho^2\delta}{[\tau(\tau+2\psi(\xi')-2)+1-\rho]^2+(\rho\delta)^2} \cdot \frac{|\xi'|^2}{(\tau+\psi(\xi')-1)^2}.
\]

Since $\rho\neq 1$, unlike the previous case of $a_\delta$, the symbol $b_\delta$ is not singular on the support of the function $\chi(\xi')\varphi(\tau)$ provided that $\epsilon_\circ$ is small enough, and this admits the uniform bound \eqref{e:symbol2} below. Making use of Lemma \ref{l:asymp} and \eqref{e:symbol2} we see that, for any $x\in B_\delta$,
\begin{align*}
	|M_\delta(D)f(x)|
	&\lesssim x_1^{-\frac{d-1}2} \int_{-2\epsilon_\circ}^{2\epsilon_\circ} \sum_{j=0}^{N-1}x_{1}^{-j} \left|\mathcal D_j b_\delta(\tau, \xi')\vert_{-\frac{x'}{|x|}}\right| +|\mathcal E_{\delta, N}(x;\tau)| d\tau \\
	&\lesssim \delta^{\frac{d-1}2} \bigg( \sum_{j=0}^{N-1}\delta^{j+1} +\delta^{N+1}\bigg), 
\end{align*}
and for a fixed large number $N$ it follows that
\begin{equation}\label{e:est-remainder}
	\|M_\delta(D)f\|_{L^q(B_\delta)}\lesssim \delta^{\frac{d-1}2+1}|B_\delta|^{\frac1q}\lesssim \delta^{\frac{d+1}2-\frac dq}.
\end{equation}
Therefore, from \eqref{e:est-main-part}, \eqref{e:est-remainder}, and the choice of the function $f\in\mathcal S(\mathbb R^d)$ in \eqref{e:choice-max}, we conclude that 
\[	\|\mathcal I(D)\|_{p\to q} \ge \frac{\|\mathcal I_\delta(D)f\|_{L^q(B_\delta)}}{\|f\|_{L^p(\mathbb R^d)}} \gtrsim \|m_\delta(D)f\|_{L^q(B_\delta)} - \|M_\delta(D)f\|_{L^q(B_\delta)} \gtrsim \delta^{\frac{d-1}2-\frac dq},
\]
which completes the proof of \eqref{e:low-im2}.
\end{proof}

Now it remains to prove the following two lemmas.
\begin{lem}\label{l:symbol}
Let $\rho\neq 1$ be a positive number and let $\psi$, $a_\delta$ and $b_\delta$ be as in the Proof of \eqref{e:low-im2}. For a fixed small number $\epsilon_\circ>0$, the following hold true: For $|\tau|\le 2\epsilon_\circ$, $|\xi'|\le \frac1{10}$, and  $0<\delta\ll 1$, we have 
\begin{align} 
\label{e:symbol}
	|\partial_{\xi'}^\alpha a_\delta(\tau, \xi')| 
		&\lesssim \frac{\delta}{\tau^2(\tau+2\psi(\xi')-2)^2+\delta^2}, \\
\label{e:symbol2}
	|\partial_{\xi'}^\alpha b_\delta(\tau, \xi')| 
		&\lesssim \frac{\delta}{[\tau(\tau+2\psi(\xi')-2)+1-\rho]^2+(\rho\delta)^2}\lesssim \delta.
\end{align}
\end{lem}
\begin{proof} The case $|\alpha|=0$ is trivial, so let us consider the case $|\alpha|\ge 1$. 

For $2\le j\le d$,
\[	\partial_j a(\tau,\xi')= \frac{-4\delta\tau^2(\tau+2\psi-2)\partial_j\psi}{(\tau^2(\tau+2\psi-2)^2+\delta^2)^2}. 	\]
Hence, the estimate \eqref{e:symbol} with $|\alpha|=1$ follows since $\epsilon_\circ$ is small and $|\tau+2\psi-2|\approx 1$. For $2\le j, k\le d$,
\[	\partial_{jk} a(\tau,\xi') = \frac{-4\delta \tau^2(2\partial_j\psi\partial_k\psi +(\tau+2\psi-2)\partial_{jk}\psi)}{(\tau^2(\tau+2\psi-2)^2+\delta^2)^2}  + \frac{32\delta \tau^4(\tau+2\psi-2)^2\partial_j\psi\partial_k\psi}{(\tau^2(\tau+2\psi-2)^2+\delta^2)^3},	\]
and the estimate \eqref{e:symbol} with $|\alpha|=2$ follows.  Next, for $2\le j,k,l \le d$, 
\begin{align*}
	&\partial_{jkl} a(\tau,\xi')
	= \frac{-4\delta\tau^2(2\p_{jk}\psi\p_l\psi+ 2\p_{kl}\psi\p_j\psi + 2\partial_{lj}\psi\partial_k\psi+ (\tau+2\psi-2)\p_{jkl}\psi)}{(\tau^2(\tau+2\psi-2)^2+\delta^2)^2} \\
	&\quad + \frac{32\delta\tau^4(\tau+2\psi-2)(6\p_j\psi\p_k\psi\p_l\psi +(\tau+2\psi-2)(\p_{jk}\psi\p_l\psi+\p_{kl}\psi\p_j\psi+\p_{lj}\psi\p_k\psi))}{(\tau^2(\tau+2\psi-2)^2+\delta^2)^3}\\
	&\quad - \frac{384\delta\tau^6(\tau+2\psi-2)^3\p_j\psi\p_k\psi\p_l\psi}{(\tau^2(\tau+2\psi-2)^2+\delta^2)^4},
\end{align*}
and this gives \eqref{e:symbol} for $|\alpha|=3$. Now, an easy induction argument shows that for any $|\alpha|\ge 1$,
\[	\p^\alpha_{\xi'}a(\tau,\xi')=\sum_{j=1}^{|\alpha|} \frac{\delta\tau^{2j}p_j(\nabla^{|\alpha|}\psi)}{(\tau^2(\tau+2\psi-2)^2+\delta^2)^{j+1}},	\]
where $\nabla^{k}\psi :=\{\p ^\beta \psi \colon |\beta|\le k\}$ for $k\in\mathbb N$ and $p_j$ is a polynomial with coefficients in $\mathbb Z\cup\{\tau\}$. Therefore, the estimate \eqref{e:symbol} follows.

The first inequality in \eqref{e:symbol2} can be proved in the same argument and we omit repetition. The second inequality in \eqref{e:symbol2} holds since $\rho\neq 1$ and 
\[	|\tau(\tau+2\psi(\xi')-2)+1-\rho|\ge|1-\rho|-2\epsilon_\circ (2+2\epsilon_\circ),	\] 
which is $\gtrsim 1$ if $\epsilon_\circ$ is small enough depending on $|1-\rho|$.
\end{proof}

We now invoke the stationary phase method (see \cite[Chapter VII]{Ho}) to obtain the  asymptotic for the function $x\mapsto I_\delta(x; \tau)$. Since $\nabla\psi(0)=0$ and $\nabla\psi(\xi')=\xi'+O(|\xi'|^3)$, by the inverse function theorem, there exists a unique diffeomorphism $g$ from the open ball $B_{d-1}(0, 1/2)$ onto an open set $U\subset\mathbb R^{d-1}$ such that $g(0)=0$ and 
\[	\eta+\nabla\psi(g(\eta))=0.\]
In fact, $g$ can be computed explicitly and we have
\begin{equation}\label{e:normal}
	g(\eta)=\frac{-\eta}{\sqrt{1+|\eta|^2}}	
\end{equation}
with $ U=B_{d-1}(0,1/\sqrt5)$. For each $\xi'\in\supp\chi$, let us denote by $K(\xi')$ the Gaussian curvature of the graph (of the unit sphere) $\mathcal G(\psi):=\{\xi\in\mathbb R^d \colon \xi_1=\psi(\xi')=1-\sqrt{1-|\xi'|^2},\, \xi'\in\supp \chi\}$ at point $(\psi(\xi'), \xi')$. Hence, $|K(\xi')|\equiv 1$. 

The following lemma is now an immediate consequence of \cite[Theorem 7.7.5 and Theorem 7.7.6]{Ho}. Also, see \cite[Lemma 2.7]{KL19}.
\begin{lem}\label{l:asymp}
	Let $0<\delta\le 1$, $-1\le \tau\le 1$, and let $I_\delta$ be as in the Proof of \eqref{e:low-im2}. If $|x_1|\ge 1/2$ and $2^5|x'|\le |x_1|$, then for every $N\in \mathbb N$ we have
\begin{align*}
	I_\delta(x;\tau)
	&=\frac{c_d}{\sqrt{|K(g(\frac{x'}{x_1}))|}} e^{i(x'\cdot g(\frac{x'}{x_1}) +x_1\psi\circ g(\frac {x'}{x_1}))} \sum_{j=0}^{N-1} \mathcal D_{j} a_\delta(\tau,\xi')\vert_{\xi'=g(\frac{x'}{x_1})}|x_1|^{-\frac{d-1}{2}-j} +\mathcal E_{\delta, N} (x;\tau) \\
	&= c_d e^{i(x_1-\sgn(x_1)|x|)} \sum_{j=0}^{N-1}\mathcal D_j a_\delta(\tau, \xi')\vert_{\xi'=-\sgn(x_1)\frac{x'}{|x|}} |x_1|^{-\frac{d-1}2-j} +\mathcal E_{\delta, N} (x;\tau),
\end{align*}
where $c_d$ is a constant depending only on $d$,  $\mathcal D_0a_\delta=a_\delta$ and, for each $j\ge1$, $\mathcal D_j$ is a differential operator in $\xi'$ of order $2j$ whose coefficients vary smoothly depending on $(\partial_{\xi'}^\alpha \psi)\circ g(\frac{x'}{x_1})$, $2\le |\alpha|\le 2j+2$. For $\mathcal E_{\delta, N}(x;\tau)$ we have the estimate
\begin{equation}\label{e:asymp-err}
	 |\mathcal E_{\delta, N}(x;\tau)|\lesssim |x_1|^{-N}\sum_{|\alpha|\le 2N} \sup_{(\tau,\xi')} |\partial_{\xi'}^\alpha a_\delta(\tau, \xi')| 
\end{equation}
with implicit constant independent of $\delta$.	
\end{lem}

\section{Failure of Carleman estimate: Proof of Theorem \ref{t:Carleman}} 
In this section, we prove Theorem \ref{t:Carleman}. Scaling consideration shows that the estimate \eqref{e:car} is possible only if 
\begin{equation}\label{e:nec1}
	\frac1p-\frac1q=\frac2d.
\end{equation}
Hence, by homogeneity, we may assume that $|v|=1$ without loss of generality. Furthermore, it is sufficient to consider $v=e_1$ only, since the Lam\'e operator is rotationally symmetric and the estimate \eqref{e:car} is invariant under any rotation $x\to Rx$, $R\in\SO(d)$. Now we shall find another necessary condition for \eqref{e:car} (with $v=e_1$) which cannot be true under the condition \eqref{e:nec1}. 

Setting $f:=e^{v\cdot x}(-\Delta^\ast)e^{-v\cdot x}u$, $f=(f_1,\ldots,f_d)^t$ and $u=(u_1,\ldots,u_d)^t$, direct calculation shows that, for $1\le j\le d$,
\[	f_j=-\mu\left(\Delta-2v\cdot \nabla +|v|^2 \right) u_j -(\lambda+\mu)\sum_{k=1}^d  (\p_j\p_k-v_j\p_k-v_k\p_j+v_jv_k )u_k.	\] 
Taking the Fourier transform we get the following identity:
\[	\widehat f(\xi) = \left( \mu \big((\xi+iv)^t(\xi+iv)\big)I_d +(\lambda+\mu) (\xi+iv)(\xi+iv)^t \right)\widehat u(\xi).	\]
Setting $\eta:=\xi+iv$ we note that the matrix $M_\eta:=\mu(\eta^t\eta)I_d+(\lambda+\mu)\eta \eta^t$ is of the form \eqref{e:matrix} in the proof of Lemma \ref{l:resol} by replacing $\xi\to \eta$ and $z\to 0$. The inverse $M_\eta^{-1}$ can be computed without difficulty  by the same manner as in the proof of \eqref{e:mult}. Thus, we get
\begin{equation}\label{e:car-sys}
	\widehat{u_j}(\xi)
	=\frac{\widehat{f_j}(\xi)}{\mu(\xi+iv)^t(\xi+iv)}
	+\bigg(\frac1{\lambda+2\mu}-\frac1{\mu }\bigg)\sum_{k=1}^d \frac{(\xi_j+iv_j)(\xi_k+iv_k) \widehat{f_k}(\xi)}{((\xi+iv)^t(\xi+iv))^2}.
\end{equation}
Let us assume \eqref{e:car} and set $f_k=0$ whenever $k\neq 2$. Since $v=e_1$, we have
\[	\widehat{u_1}(\xi)=\bigg(\frac1{\lambda+2\mu}-\frac1{\mu }\bigg) \frac{(\xi_1+i)\xi_2\widehat{f_2}(\xi)}{(|\xi|^2-1+2i\xi_1)^2},	\]
and the inequality \eqref{e:car} implies 
\begin{equation}\label{e:car-lame}
	\bigg\|\mathcal F^{-1}\bigg(\frac{(\xi_1+i)\xi_2\widehat{h}(\xi)}{(|\xi|^2-1+2i\xi_1)^2}\bigg) \bigg\|_{L^q(\mathbb R^d)} \lesssim \|h\|_{L^p(\mathbb R^d)}.
\end{equation}

Now we show that \eqref{e:car-lame} is possible only if
\begin{equation}\label{e:nec2}
	\frac1p-\frac1q \ge \frac4{d+2}. 
\end{equation}

To see this let us fix nonnegative functions $\phi\in C_c^\infty((1/2, 2))$ and $\psi\in C_c^\infty([0,2))$ such that $\phi=1$ on $[2/3,3/2]$ and $\psi=1$ on $[0,1]$. Then for every $\delta>0$ small enough, let us define $h_\delta\in \mathcal S(\mathbb R^d)$ by
\[	\widehat{h_\delta}(\xi)=\phi\Big(\frac{\xi_1}{\delta}\Big) \phi\Big(\frac{\xi_2-1}{\delta}\Big) \psi\Big(\frac{|\bar\xi|}{\sqrt{\delta}}\Big),	\]
where  $\bar\xi:=(\xi_3,\ldots,\xi_d)\in\mathbb R^{d-2}$. Note that on the support of  $\widehat{h_\delta}$ we have 
\[	||\xi|^2-1+2i\xi_1|^2
	= (\xi_1^2+(\xi_2+1)(\xi_2-1)+|\bar\xi|^2)^2+4\xi_1^2 
	\le (4\delta^2 + 2\delta(2+2\delta)+4\delta)^2 +16\delta^2 \lesssim \delta^2.
\]
If we define the set $A_\delta$ by
\[	A_\delta:=\big\{x\in\mathbb R^d: |x_1|\le (100\delta)^{-1},~ |x_2|\le (100\delta)^{-1}, ~|\bar x|\le (100\sqrt\delta)^{-1}\big\},	\]
then the estimate \eqref{e:car-lame} with $h=h_\delta$  yields 
\[	|A_{\delta}|^{\frac 1q} \delta^{-2}|\supp \widehat{h_\delta}| \lesssim \|h_\delta\|_{L^p(\mathbb R^d)}	\]
with the implicit constant independent of $\delta\ll 1$. This implies $\delta^{\frac{d-2}{2}} \delta^{-\frac{d+2}{2q}}\lesssim \delta^\frac{d+2}2 \delta^{-\frac{d+2}{2p}}$ or, equivalently, $\delta^{-2+\frac{d+2}{2}(\frac1p-\frac1q)}\lesssim 1$ which implies \eqref{e:nec2}. We have shown that the estimate \eqref{e:car} implies both \eqref{e:nec1} and \eqref{e:nec2}. Combining these results we conclude that  \eqref{e:car} is impossible other than the case $(d,p,q)=(2,1,\infty)$.

It remains to show the failure of \eqref{e:car-lame}  with $(d,p,q)=(2,1,\infty)$.  Let $v=(1,0)$ and let $(f_1, f_2)=(0, h_\epsilon)$, where $h_\epsilon$ is the function defined in \eqref{e:def_h}. Then \eqref{e:car-sys} gives 
\[	(2\pi)^2 u_2(0) =\frac1\mu\int_{\mathbb R^2} \frac{\widehat {h_\epsilon}(\xi)}{|\xi|^2+2i\xi_1-1}d\xi + \left( \frac1{\lambda+2\mu}-\frac1\mu\right) \int_{\mathbb R^2}\frac{\xi_2^2\widehat{h_\epsilon}(\xi)}{(|\xi|^2+2i\xi_1-1)^2}d\xi.	\]
Since we want to disprove the $L^1$--$L^\infty$ bound and $\{ h_\epsilon \colon \epsilon>0\}$ is uniformly bounded in $L^1(\mathbb R^2)$ (see \eqref{e:unif_norm}), as in \emph{Proof of Proposition \ref{p:admissible}},  it is sufficient to show that $\real u_2(0)\to \infty$ as $\epsilon\to 0$.

We note that 
\begin{align*}
	(2\pi)^2\real u_2(0)
	&= \frac1\mu \int \frac{(\xi_1^2-1)[(|\xi|^2-1)^2 -(2\xi_1)^2]+8\xi_1^2(|\xi|^2-1)}{[(|\xi|^2-1)^2+(2\xi_1)^2]^2} \widehat{h_\epsilon}(\xi)d\xi \\
	& \quad + \frac{1}{\lambda+2\mu} \int \frac{\xi_2^2 [(|\xi|^2-1)^2-(2\xi_1)^2]}{[(|\xi|^2-1)^2+(2\xi_1)^2]^2} \widehat{h_\epsilon}(\xi) d\xi,
\end{align*}
which we rearrange as follows:
\begin{align*}
	&\frac1\mu \int \frac{\xi_1^2 [(|\xi|^2-1)^2-(2\xi_1)^2]}{[(|\xi|^2-1)^2+(2\xi_1)^2]^2} \widehat{h_\epsilon}(\xi) d\xi +\frac{1}{\lambda+2\mu} \int \frac{\xi_2^2 [(|\xi|^2-1)^2-(2\xi_1)^2]}{[(|\xi|^2-1)^2+(2\xi_1)^2]^2} \widehat{h_\epsilon}(\xi) d\xi \\
	& + \frac1\mu \int \frac{ 8\xi_1^2(|\xi|^2-1) - [(|\xi|^2-1)^2-(2\xi_1)^2]}{[(|\xi|^2-1)^2+(2\xi_1)^2]^2} \widehat{h_\epsilon}(\xi) d\xi.
\end{align*}
Easy computations show that for $|\xi|\ge4$,
\begin{gather}
	|\xi|^4/2 \le (|\xi|^2-1)^2-(2\xi_1)^2 \le |\xi|^4, 	\label{e:est1} \\
	|\xi|^4/2 \le (|\xi|^2-1)^2 + (2\xi_1)^2 \le 2|\xi|^4,		\label{e:est2} 
\end{gather}
and
\begin{equation} \label{e:est3}
	0\le 8\xi_1^2(|\xi|^2-1) \le 8 |\xi|^4.
\end{equation}
Let $m=\min\{\frac1\mu, \frac1{\lambda+2\mu}\}$. Since $\widehat{h_\epsilon}$ is non-negative and satisfies  \eqref{e:supp} and \eqref{e:supp1}, it follows from the triangle inequality and the estimates \eqref{e:est1}, \eqref{e:est2}, and \eqref{e:est3} that if  $\epsilon\le 1/8$ then
\begin{align*}
	|(2\pi)^2 \real u_2(0)|
	&\ge m \int \frac{|\xi|^2 [(|\xi|^2-1)^2-(2\xi_1)^2]}{[(|\xi|^2-1)^2+(2\xi_1)^2]^2} \widehat{h_\epsilon}(\xi)d\xi \\
	& \quad - \frac1\mu  \int \frac{|8\xi_1^2(|\xi|^2-1)-[(|\xi|^2-1)^2-(2\xi_1)^2]|}{[(|\xi|^2-1)^2+(2\xi_1)^2]^2} \widehat{h_\epsilon}(\xi) d\xi \\
	&\gtrsim \int_{1/\epsilon\le |\xi|\le 1/2\epsilon^2} |\xi|^{-2} d\xi - \int_{1/2\epsilon\le |\xi|\le 1/\epsilon^2} |\xi|^{-4} d\xi \\
	&\gtrsim \log \frac1\epsilon - \epsilon^2,
\end{align*}
which diverges to infinity as $\epsilon \to 0$.

\end{document}